\documentclass[11pt]{article}
\usepackage[utf8]{inputenc}
\usepackage[T1]{fontenc}

\usepackage{enumerate}
\usepackage{url}
\usepackage{epsf}
\usepackage{epsfig}
\usepackage{graphicx}
\usepackage{subfig}
\usepackage{array}
\usepackage{verbatim}
\usepackage{soul}

\usepackage{amsmath,amsfonts,amssymb,amsthm}
\usepackage{wasysym}

\usepackage{stmaryrd}

\usepackage{ gensymb }
\usepackage{ esint }

\usepackage{hyperref}

\def\cA{{\cal A}}

\def\cC{{\cal C}}

\def\cE{{\cal E}}
\def\cF{{\cal F}}

\def\cH{{\cal H}}

\def\cO{{\cal O}}
\def\cP{{\cal P}}

\def\cS{{\cal S}}

\def\mD{{\mathbb D}}

\def\mH{{\mathbb H}}

\def\mM{{\mathbb M}}
\def\mN{{\mathbb N}}

\def\mP{{\mathbb P}}
\def\mQ{{\mathbb Q}}
\def\mR{{\mathbb R}}
\def\mS{{\mathbb S}}

\def\mZ{{\mathbb Z}}

\def\R{\mR}

\def\Z{\mZ}

\def\ve{\varepsilon}

\DeclareMathOperator\Id{Id}

\def\<{\langle}
\def\>{\rangle}

\DeclareMathOperator\diver{div}

\def\sm{\setminus}

\makeatletter
\newcommand{\oset}[2]{%
  {\mathop{#2}\limits^{\vbox to -.5\ex@{\kern-\tw@\ex@
   \hbox{\scriptsize #1}\vss}}}}
\makeatother

\renewcommand\subset{\subseteq}

\DeclareMathOperator\Hull{Hull}

\newtheorem{theorem}{Theorem}[section]
\newtheorem{remark}[theorem]{Remark}
\newtheorem{lemma}[theorem]{Lemma}
\newtheorem{proposition}[theorem]{Proposition}
\newtheorem{corollary}[theorem]{Corollary}
\newtheorem{definition}[theorem]{Definition}

\newtheorem*{proposition*}{Proposition}
\newtheorem*{theorem*}{Theorem}

\DeclareMathOperator{\bary}{bary}

\usepackage{algorithm,algpseudocode,algorithmicx}

\def\RZ2{\cF(\Z^2)}

\DeclareMathOperator\Leb{Leb}

\DeclareMathOperator\SDiff{SDiff}
\DeclareMathOperator\eval{e}
\DeclareMathOperator\Prob{Prob}
\DeclareMathOperator\BV{BV}

\author{
Quentin Mérigot\thanks{CNRS, Université Paris-Dauphine, UMR 7534, CEREMADE, Paris, France.}
\thanks{ANR grant TOMMI, ANR-11-BS01-014-01}
\and
Jean-Marie Mirebeau{\footnotemark[1]}~\thanks{
ANR grant NS-LBR, ANR-13-JS01-0003-01
}
}

\DeclareMathOperator\supp{supp}

\begin{document}
\title{
Minimal geodesics along volume preserving maps,\\
through semi-discrete optimal transport
}
\maketitle
\date{}

\begin{abstract}
We introduce a numerical method for extracting minimal geodesics along the group of volume preserving maps, equipped with the $L^2$ metric, which as observed by Arnold \cite{Arnold:1966tf} solve Euler's equations of inviscid incompressible fluids. The method relies on the generalized polar decomposition of Brenier \cite{Brenier:1991fz}, numerically implemented through semi-discrete optimal transport. 
It is robust enough to extract non-classical, multi-valued solutions of Euler's equations, for which the flow dimension is higher than the domain dimension, a striking and unavoidable consequence of this model \cite{Shnirelcprimeman:1994ef}. Our convergence results encompass this generalized model, and our numerical experiments illustrate it for the first time in two space dimensions.
\end{abstract}

\section{Introduction}

The motion of an inviscid incompressible fluid, moving in a compact domain $X \subset \mR^d$, is described by Euler's \cite{Anonymous:uz} equations 
\begin{align}
\label{eqdef:Euler}
	\partial_t v + (v\cdot \nabla ) v &= - \nabla p & \diver v &=0,
\end{align}
coupled with the impervious boundary condition $v \cdot n=0$ on $\partial \Omega$. Here $v$ denotes the fluid velocity, and $p$ the pressure acts as a Lagrange multiplier for the incompressibility constraint.
In Lagrangian coordinates, Euler equations \eqref{eqdef:Euler} yield the geodesic equations along the group $\SDiff$ volume preserving diffeomorphisms of $X$, equipped with the $L^2$ metric \cite{Arnold:1966tf}.
Consider an inviscid incompressible fluid flowing during the time interval $[0,1]$, and a map $s^* : X \to X$ giving the final position $s^*(x)$ of each fluid particle initially at position $x\in X$. 
In this paper, we 
discretize and 
numerically investigate a natural approach to reconstruct the  intermediate fluid states: look for a minimizing geodesic joining the initial configuration $s_*=\Id$ to the final one $s^*$
\begin{align}
\label{eqdef:MinimizeSDiff}
\text{minimize} &\int_0^1 \|\dot s(t)\|^2 dt, & \text{subject to } s(0)=s_*,\ s(1)=s^*,\text{ and } \forall t \in [0,1], \ s(t) \in \mS. 
\end{align}
We denoted by $\mS \subset L^2(X, \mR^d)$ the space of maps preserving the Lebesgue measure on $X$, which in dimension $d \geq 2$ is the closure of $\SDiff$. 
Despite this first relaxation, note that the optimized functional in \eqref{eqdef:MinimizeSDiff} does not penalize the spatial derivatives of $s$, whereas the constraint involves the jacobian of $s$. 
The study of \eqref{eqdef:MinimizeSDiff} thus requires non-standard variational techniques, reviewed in \cite{Figalli:2012tm}. 
 
In dimension $d \geq 3$, the optimization problem \eqref{eqdef:MinimizeSDiff} needs not have a minimizer in $s \in H^1([0,T], \mS)$  \cite{Shnirelcprimeman:1994ef}, and minimizing sequences $(s_n)_{n \in \mN}$ may instead display oscillations reminiscent of an homogeneization phenomenon. 
A second relaxation is required, based on generalized flows \cite{Brenier:1985bz} which allow  particles to split and their paths to cross. This surprising behavior is an unavoidable counterpart of the lack of viscosity in Euler's equations, which amounts to an infinite Reynolds number. 
Generalized flows are also relevant in dimension $d \in \{1,2\}$ if the underlying physical model actually involves a three dimensional domain $X \times [0, \ve]^{3-d}$ in which one neglects the fluid acceleration in the extra dimensions \cite{Brenier:2008ho}.
Consider the space of continuous paths (of fluid particles)
\begin{equation*}
	\Omega := C^0([0,1],X).
\end{equation*}
Let $e_t(\omega):=\omega(t)$ be the evaluation map at time $t\in [0,1]$, so that $(e_0,e_1) (\omega) = (\omega(0), \omega(1))$. Let also $\Leb$ denote the Lebesgue measure restricted to the domain $X$, normalized for unit mass, and let $f \# \mu$ denote the push-forward of a measure $\mu$ by a measurable map $f$. 
The geodesic distance \eqref{eqdef:MinimizeSDiff} admits a convex relaxation, linearizing both the objective and the constraints, and for which the existence of a minimizer is guaranteed. It is posed on probability measures on $\Omega$, called \emph{generalized flows}
\begin{align}
\label{eqdef:MinimizeRelaxed}
	d^2(s_*,s^*) &:=\min_{\mu \in \Prob(\Omega)} \int_\Omega \cA(\omega) d \mu(\omega),
	&
	\text{subject to} &
	\begin{cases}
		\cA(\omega) := \int_0^1|\dot\omega(t)|^2 dt\\
		(\eval_0,\eval_1)\#\mu = (s_*,s^*)\#\Leb,\\
		\forall t \in [0,1], \ \eval_t\# \mu = \Leb.
	\end{cases}
%
\end{align}
Note that the path action $\cA : \Omega \to \mR_+ \cup \{+\infty\}$, although unbounded, is lower semi-continuous.
The first constraint $(\eval_0,\eval_1)\#\mu = (s_*,s^*)\#\Leb$ expresses that moving fluid particles from $s_*(x)$ to $s^*(x)$ for all $x \in X$, or from the origin $\omega(0)$ to the end $\omega(1)$ of the paths $\omega \in \Omega$ as weighted by $\mu$, yields equivalent transport plans. The second constraint $\eval_t\# \mu = \Leb$ states that the path positions $\omega(t)$, as weighted by $\mu$, equidistribute on $X$ at each time $t \in [0,1]$, which amounts to incompressibility. 
A classical flow $s \in H^1([0,1], \mS)$ can be regarded as a generalized flow, with paths $t \mapsto s(t,x)$, weighted by the Lebesgue measure on $x \in X$.
Our discretization truly solves \eqref{eqdef:MinimizeRelaxed}, rather than \eqref{eqdef:MinimizeSDiff}, and convergence is established in this relaxed setting.

\begin{figure}
\centering
	\includegraphics[height=4cm]{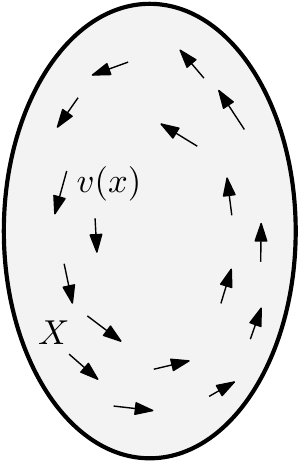}\hspace{1cm}
	\includegraphics[height=4cm]{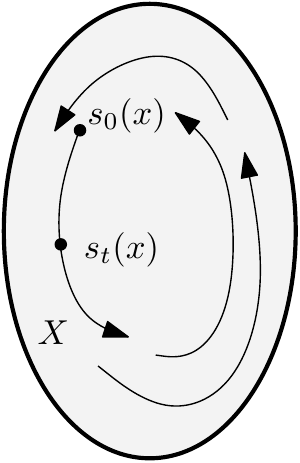}
	\hspace{1cm}
	\includegraphics[height=4cm]{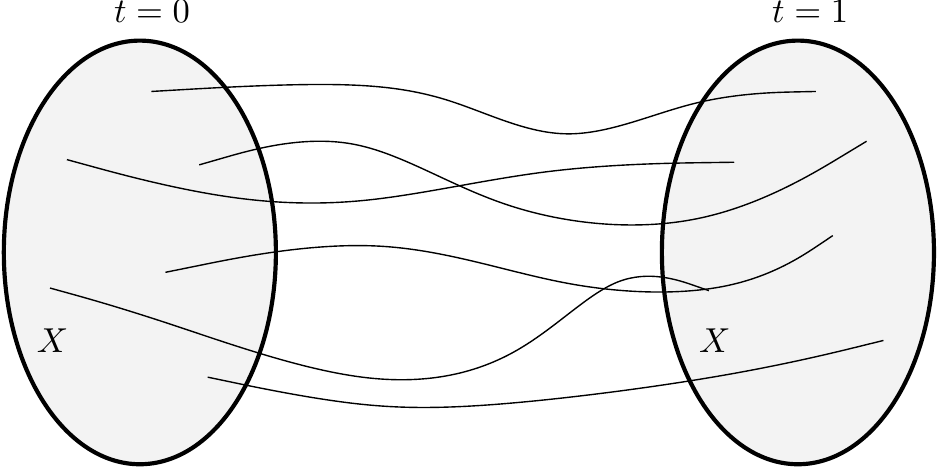}
	\caption{
	The motion of inviscid incompressible fluids admits three formulations, either (I) Eulerian based on the local speed $v : [0,1]\times X \to \mR^d$, (II) Lagrangian based on diffeomorphisms $s(t, \cdot)$ which integrate the speed: $\partial_t s(t,x) = v(t,s(t,x))$, or (III) relaxed as a superposition of individual  particles paths $\omega \in \Omega$, weighted by a measure $\mu$.
	}
\end{figure}

The incompressibility constraint in \eqref{eqdef:Euler}, \eqref{eqdef:MinimizeSDiff} and \eqref{eqdef:MinimizeRelaxed}, gives rise to a Lagrange multiplier, the pressure, which is the \emph{unique} maximizer to a concave optimization problem dual to \eqref{eqdef:MinimizeRelaxed}, see \cite{Brenier:1993wd}. The primal \eqref{eqdef:MinimizeRelaxed} may in contrast have several solutions, up to the notable exception \cite{Bernot:2009ed} of smooth flows in dimension $d=1$. 
The pressure is a classical function $p \in L^2_{\rm loc}(\,]0,T[ ,\, \BV(X))$, see \cite{Ambrosio:2007jk} (which requires the technical assumption that $X$ is a $d$-dimensional torus). This regularity is sufficient to show that any solution $s$ to \eqref{eqdef:MinimizeSDiff} (resp.\ $\mu$-almost any path $\omega$, for any solution $\mu$ to \eqref{eqdef:MinimizeRelaxed}) satisfies 
\begin{align}	
\label{eq:PathDynamics}
	\partial_{tt} s(t,x) &= -\nabla p(t,s(t,x)), & \text{resp.} \quad  \ddot \omega(t) = -\nabla p(t,\omega(t)). 
\end{align}
In other words, fluid particles move by inertia, only deflected by the force of pressure. 
Assume that the pressure hessian is sufficiently small, precisely that
\begin{equation}
\label{eq:PressureClassical}
\forall t\in[0,1], \ \forall x \in X, \ 
\nabla^2 p(t,x) \prec \pi^2 \Id
\end{equation}
in the sense of symmetric matrices. Then using the path dynamics equation \eqref{eq:PathDynamics} Brenier \cite{Brenier:1985bz} showed that the relaxed problem \eqref{eqdef:MinimizeRelaxed} admits a unique minimizer $\mu \in \Prob(\Omega)$, which is \emph{deterministic}: in other words associated to a, possibly non-smooth but otherwise classical, minimizer $s \in H^1([0,T], \mS)$ of \eqref{eqdef:MinimizeSDiff}. Inequality \eqref{eq:PressureClassical} is sharp, and several families of examples are known for which uniqueness and/or determinism are lost precisely when the threshold \eqref{eq:PressureClassical} is passed. We present \S \ref{sec:Numerics} the first numerical illustration of this phenomenon.

\subsection{Numerical scheme and main results}

\begin{figure}
\centering
	\includegraphics[height=2.7cm]{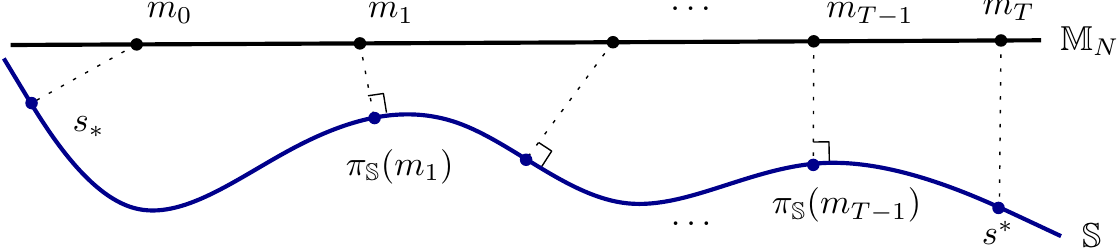}
	\caption{The geodesic distance $d^2(s_*,s^*)$ along the ``manifold'' $\mS$ of volume preserving maps, represented as a blue curve, is estimated \eqref{eqdef:DiscreteOptimization} as the length of a chain $(m_i)_{i=0}^T$ in the linear subspace $\mM_N$, represented as a black line, plus penalizations for the boundary values and the distance from the chain elements to $\mS$.}
	\label{fig:ProjMToS}
\end{figure}


We introduce a new discretization for the relaxation \eqref{eqdef:MinimizeRelaxed} of the shortest path formulation \eqref{eqdef:MinimizeSDiff} of Euler equations \eqref{eqdef:Euler}. Our approach is numerically tractable in dimension $2$, and is the first to illustrate the transition between classical and generalized solutions occurring at the threshold \eqref{eq:PressureClassical} on the pressure regularity. 

For that purpose we need to introduce some notation. Let $\mM := L^2(X, \mR^d)$, and let $\mS \subset \mM$ be the collection of maps preserving the restriction to $X$ of the Lebesgue measure, denoted by $\Leb$ and normalized to have mass $1$. For each $N \in \mN$ let $\cP_N$ be a partition of $X$ into $N$ regions of equal area $1/N$, diameter $\leq C_\cP N^{-\frac 1 d}$, and let $\mM_N \subset \mM$ be the $N$-dimensional subspace of functions which are piecewise constant on this partition. Given $s_*, s^* \in \mS$, discretization parameters $T,N \in \mN$, and a penalization factor $\lambda \gg 1$, we solve
\begin{equation}
\label{eqdef:DiscreteOptimization}
\cE(T,N,\lambda):=
	\min_{m \in \mM_N^{T+1}} 
	T \sum_{0 \leq i < T} \|m_{i+1}-m_i\|^2 
	+ \lambda\biggl(\|m_0-s_*\|^2 + \|m_T-s^*\|^2
	+  \sum_{1 \leq i < T} \inf_{s \in \mS} \|m_i-s\|^2 \biggr).
\end{equation}
In all this paper, $\|\cdot\|$ stands for the $L^2(X, \mR^n)$ norm, and $|\cdot|$ for the euclidean norm on $\mR^n$, for any $n \in \mN$.
Comparing this with \eqref{eqdef:MinimizeSDiff}, we recognize the standard discretization of the length of the discrete path $(m_0, \cdots, m_T)$, as well as an implementation by penalization of the boundary value constraints and of the incompressibility constraints. The optimization of \eqref{eqdef:DiscreteOptimization}, seen as a function of $m \in \mM_N^{T+1}$, is an $N (T+1)d$-dimensional smooth optimization problem. A quasi-Newton method gave convincing results, see \S \ref{sec:Numerics}, despite the non-convexity of the functional which forbids to guarantee that its global minimum is numerically found.


Before entering the analysis of \eqref{eqdef:DiscreteOptimization}, let us emphasize that the inner-subproblems, the projection of each $m_i$ onto the set $\mS$ of measure preserving maps, are numerically tractable thanks to two main ingredients: Brenier's polar factorization \cite{Brenier:1991fz}, and semi-discrete optimal transport. The former states that the distance from any given $m \in \mM$ to the set $\mS$, is the cost of the transport plan needed to equidistribute on $X$ the image measure of $m$
\begin{equation}
\label{eq:DistanceToIncompressibles}
	\inf_{s \in \mS} \|m-s\|^2  = W_2^2( m \# \Leb, \Leb),
\end{equation}
where $W_2$ is the Wasserstein distance for the quadratic transport cost.
If $m \in \mM_N$, then $m \# \Leb$ is the sum of $N$ Dirac measures of mass $1/N$, located at the $N$ values of the piecewise constant map $m$ on the partition $\cP_N$. 
Semi-discrete optimal transport \cite{Aurenhammer:1998ie,Merigot:2011js,Levy:2014un} is a numerical method for computing \eqref{eq:DistanceToIncompressibles}, and more generally the Wasserstein distance between a discrete measure $\eta = \sum_{j=1}^N \eta_j \delta_{x_j}$, and an absolutely continuous measure $\nu = \rho\Leb$, with a (typically) piecewise linear density $\rho$. It is based on Kantorovitch duality 
\begin{align}
	\label{eq:Kantorovitch}
	W_2^2(\eta, \nu) &= \sup_{f \in L^1(\eta)} \int_X f\eta + \int_X g \nu,  & &\text{where } \forall y \in X,\  g(y) = \inf_{x \in X} |y-x|^2 - f(x), \\
	\label{eq:KantorovitchDiscrete}
	& = \sup_{f \in \mR^N} \sum_{1 \leq j \leq N} \eta_j f_j + \int_X g \nu,  & &\text{where } \forall y \in X,\  g(y) = \min_{1 \leq j \leq N} |y-x_j|^2 - f_j,
\end{align}
where \eqref{eq:KantorovitchDiscrete} is obtained from \eqref{eq:Kantorovitch} by setting $f_j = f(x_j)$. Importantly, the conjugate $g$ in \eqref{eq:KantorovitchDiscrete} is piecewise quadratic on a partition of $X$, called the Laguerre Diagram of the sites $x_j$ with weights $f_j$, that is constructible through computational geometry software \cite{cgal}. 
The $N$-dimensional concave maximization problem \eqref{eq:KantorovitchDiscrete}, which is unconstrained and twice continuously differentiable, is 
efficiently solved via Newton or quasi-Newton methods. Semi-discrete optimal transport has become a reliable and efficient building block for PDE discretizations \cite{Benamou:2014ud}.

A second interpretation of the optimization problem
\eqref{eqdef:DiscreteOptimization}, closer to
\eqref{eqdef:MinimizeRelaxed}, involves a generalized flow $\mu \in
\Prob(\Omega)$ supported on $N$ trajectories, each piecewise linear
with direction changes at times $\{0,1/T, \cdots, T/T\}$. Let
$m=(m_i)_{i=0}^T \in \mM_N^{T+1}$, and for each $1 \leq j \leq N$ let
$m_i^j$ be the constant value of $m_i$ on the $j$-th region of the
partition $\cP_N$ of $X$. For each $1 \leq j \leq N$ let $\omega_j \in
\Omega$ be the piecewise linear path with value $m_i^j$ at time $i/T$,
for all $0 \leq i \leq T$ (see
Figure~\ref{fig:SolutionToCurves}). Finally let $\mu \in
\Prob(\Omega)$ be the discrete probability measure equidistributed on
the set of paths $\{\omega_j; \ 1 \leq j \leq N\}$. Then
\eqref{eqdef:DiscreteOptimization} rewrites in a form close to
\eqref{eqdef:MinimizeRelaxed}
\begin{equation}
\label{eq:DiscreteRelaxed}
	\int_\Omega \cA(\omega) d \mu(\omega) + 
	\lambda\Bigg(\int_X |m_0(x)-s_*(x)|^2+|m_T(x) - s^*(x)|^2 dx + 
	\sum_{\substack{1 \leq i < T\\ t=i/T}} W_2^2(e_t \# \mu, \Leb) \Bigg).
\end{equation}
Indeed, the first energy term satisfies 
\begin{equation*}
	\int_\Omega \cA(\omega) d\mu(\omega) = \frac 1 N \sum_{1 \leq j \leq N} \int_0^1 | \dot \omega_j(t)|^2 dt = \frac T N \sum_{\substack{0 \leq i < T\\1 \leq j \leq N}} |m_{i+1}^j-m_i^j|^2 = T \sum_{0 \leq i < T} \|m_{i+1}-m_i\|^2.
\end{equation*}
The penalized integral term in \eqref{eq:DiscreteRelaxed} equals $\|m_0 - s_*\|^2+\|m_T-s^*\|^2$ from \eqref{eqdef:DiscreteOptimization}. It is the cost of the transport plan on $X^2$ mapping $(m_0(x), m_T(x))$ to $(s_*(x), s^*(x))$ for all $x \in X$, which sends $(\eval_0,\eval_1)\# \mu = (m_0,m_T)\# \Leb$ onto $(s_*,s^*)\#\Leb$, and thus enforces the proximity of these two couplings on $X^2$ as required in \eqref{eqdef:MinimizeRelaxed}. The other penalized terms $W_2^2(\eval_t \# \mu, \Leb)$ account for the incompressibility of $\mu$ at time $t=i/T$, $1 \leq i < T$, and by \eqref{eq:DistanceToIncompressibles} are equal to $\inf_{s \in \mS} \|m_i-s\|^2$ from \eqref{eqdef:DiscreteOptimization}.

\begin{figure}
\centering
	\includegraphics[height=4cm]{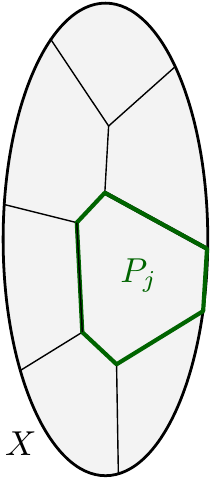}
	\hspace{1cm}
	\includegraphics[height=4cm]{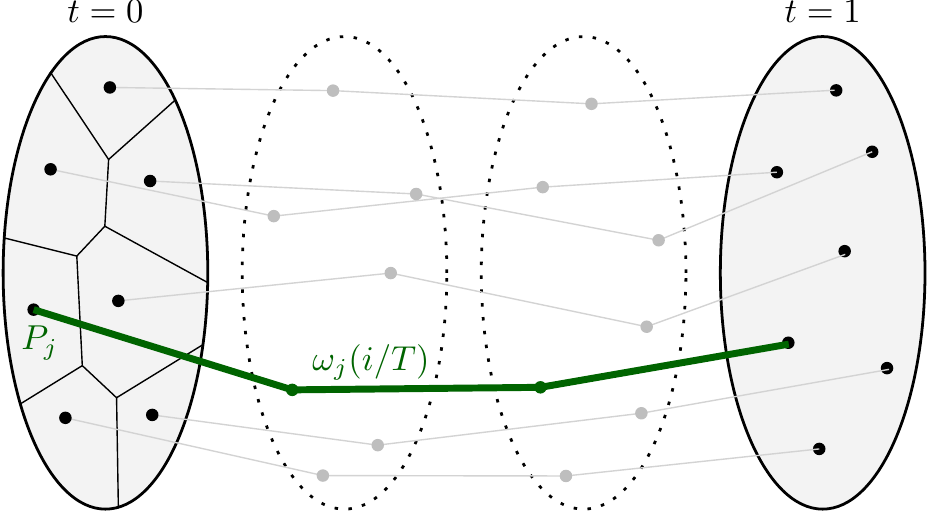}
	\caption{(Left) A partition $\cP_N$ cuts the domain $X$ into $N$ region of equal area and roughly isotropic shape. (Right) To a sequence $(m_i)_{i=0}^T \in \mM_N^{T+1}$ one can associate $N$ piecewise linear paths $(\omega_j)_{j=1}^N$, by interpolating the map values at the times $\{0,1/T, \cdots, T/T\}$ for each region of the partition $\cP_N$.\label{fig:SolutionToCurves}}
\end{figure}

Summarizing, the geodesic formulation of Euler equations \eqref{eqdef:MinimizeSDiff} has a rather surprising relaxation \eqref{eqdef:MinimizeRelaxed}, looking a-priori unphysical: fluid particles may split and cross. Yet the natural discretizations \eqref{eqdef:DiscreteOptimization} and \eqref{eq:DiscreteRelaxed} of these two formulations are actually \emph{identical}. The classical and generalized interpretations are also at the heart of our main result.
\begin{theorem}
\label{th:UpperBound}
	Let $s_*,s^* \in \mS$, let $T,N \in \mN$, and $\lambda \geq 0$. 
	The relaxed geodesic distance \eqref{eqdef:MinimizeRelaxed} 
	and the discretized minimum \eqref{eqdef:DiscreteOptimization} 
	 satisfy 
	\begin{equation*}
		\cE(T,N, \lambda) \leq d(s_*,s^*) + \cO(T h_N^2 \lambda)	,
	\end{equation*}
	\begin{itemize}
		\item (Classical estimate)  with $h_N = N^{-\frac 1 d}$, if the classical geodesic distance \eqref{eqdef:MinimizeSDiff} equals the relaxed distance \eqref{eqdef:MinimizeRelaxed}, and admits a minimizer with regularity $s \in L^\infty([0,1], H^1(X))$.
		\item (Relaxed estimate) with $h_N = N^{-\frac 1 {2d}}$ (resp. $N^{-\frac 1 2} \sqrt {\ln N}$ if $d=1$), if the pressure field gradient $\nabla p$ is Lipschitz on $[0,1] \times X$, and the boundary data $s_*^{-1}, s^*$ are Lipschitz on $X$.
	\end{itemize}
\end{theorem}
Recall that the classical \eqref{eqdef:MinimizeSDiff} and relaxed \eqref{eqdef:MinimizeRelaxed} distances are automatically equal in dimension $d \geq 3$, and that the pressure field gradient $\nabla p$ is uniquely determined by the boundary values $s_*,s^*$.

The decay rate $h_N = N^{-\frac 1 D}$ in Theorem \ref{th:UpperBound} is actually tied to the dimension $D_{\rm quant}(\mu)$ of the generalized flow $\mu \in \Prob(\Omega)$ minimizing \eqref{eqdef:MinimizeRelaxed}, see Definition \ref{def:Dimension}. The flow associated to a classical solution has dimension $d$, since the particle trajectories are determined by their initial position $x \in X \subset \mR^d$. The trajectories of a generalized flow obey a second order ordinary differential equation \eqref{eq:PathDynamics} and are thus determined by their initial position and velocity $(x,v) \in X \times \mR^d \subset \mR^{2d}$, provided Cauchy-Lipschitz's theorem applies.
The generalized flow dimension is thus $2d$ in the worst case, but intermediate dimensions $d < D < 2d$ are also common, see \S \ref{sec:Numerics}.

Theorem \ref{th:UpperBound} does not tell how to choose the constraint penalization parameter $\lambda$. The next proposition shows that the quantity $\cE'(T,N,\lambda) := (1+4 T/\lambda) \cE(T,N,\lambda)$ arises naturally in error estimates, which suggests to choose $\lambda = h_N^{-1} = N^{\frac 1 D}$ so that 
\begin{equation}
\label{eq:ChoiceLambda}
	\cE'(T,N,\lambda) 
	= d^2(s_*,s^*) + \cO(T/\lambda + T h_N^2 \lambda) = d^2(s_*,s^*) + \cO(T N^{-\frac 1 D}).
\end{equation}

\begin{proposition}
\label{prop:Construction}
Let $m=(m_i)_{i=0}^T\in \mM_N^{T+1}$ be a minimizer of \eqref{eqdef:DiscreteOptimization}. 
\begin{itemize}
	\item (Classical construction) Let $(s_i)_{i=0}^T$ be the chain of incompressible maps defined by: $s_0 = s_*$, $s_T=s^*$, and $s_i$ is a projection of $m_i$ onto $\mS$ for all $1\leq i < T$. Then 
	\begin{equation*}
		T \sum_{0 \leq i < T} \|s_{i+1}-s_i\|^2 \leq \cE'(T,N,\lambda). 
	\end{equation*}
	\item (Relaxed construction) Let $\mu \in \Prob(\Omega)$ be the generalized flow built from $(m_i)_{i=0}^T$ as in \eqref{eq:DiscreteRelaxed}. Then
	\begin{equation*}
		\int_0^1 W_2^2(\eval_t \# \mu,\Leb) \, dt \leq \frac 1 {4 T^2}  \cE'(T,N,\lambda). 
	\end{equation*}
	As a result, let $(N_T, \lambda_T)_{T \in \mN}$ be such that $\cE'(T,N_T,\lambda_T) \to d(s_*,s^*)$ as $T \to \infty$. Then a subsequence of the associated flows $(\mu_T)_{T \in \mN}$ weak-* converges to a minimizer of \eqref{eqdef:MinimizeRelaxed}.
\end{itemize}
\end{proposition}

\paragraph{Outline.} Theorem \ref{th:UpperBound} is established \S \ref{subsec:ClassicalUpperBound} (Classical estimate) and \S \ref{sec:Relaxed} (Relaxed estimate). Proposition \ref{prop:Construction} is proved \S \ref{subsec:IncompressibleChain}. Numerical experiments are presented \S \ref{sec:Numerics}.

\begin{remark}[Monge-Ampere gravitation]
  Some models of reconstruction of the early universe
  \cite{Brenier:2011hs} involve actions of a form closely related to
  our discrete energy functional \eqref{eqdef:DiscreteOptimization},
  for the parameter value
  $\lambda=2$: 
\begin{equation*}
	\int_{0}^{1} \left(\frac 1 2 \|\dot m(t)\|^2 + \inf_{s \in \mS} \|m(t)-s\|^2\right) dt. 
\end{equation*}
\end{remark}

\section{Classical analysis}
\label{sec:Classical}

We establish Theorem \ref{th:UpperBound} (Classical estimate) in \S \ref{subsec:ClassicalUpperBound}, and prove Proposition \ref{prop:Construction} in \S \ref{subsec:IncompressibleChain}. The optimization parameters $(T,N,\lambda, s_*,s^*)$ are fixed in this section.

\subsection{Upper estimate of the discretized energy}
\label{subsec:ClassicalUpperBound}

Following the assumption of Theorem \ref{th:UpperBound} (Classical estimate), we consider a minimizer of the shortest path problem \eqref{eqdef:MinimizeSDiff}, and assume that it has regularity $s \in L^\infty([0,1], H^1(X))$.
Define $s_i := s(i/T)$ for all $0 \leq i \leq T$, and note that $s_0=s_*$, $s_T=s^*$. Let also $m_i := \mP_N (s_i)$, for all $0 \leq i \leq T$, where $\mP_N : \mM \to \mM_N$ denotes the orthogonal projection. 
We denote by $h_N := N^{-\frac 1 d}$ the discretization scale, and recall that each region of the partition $\cP_N$ of $X$ has area $1/N$ and diameter $\leq C_\cP h_N$.

Let $s_i^P$ denote the mean of $s_i$ on the region $P$ of the partition $\cP_N$, for all $0 \leq i \leq T$. Then
\begin{align}
\label{eq:msClassical}
	\|s_i-m_i\|^2  = \sum_{P \in \cP_N} \int_P |s_i(x)-s_i^P|^2 dx &\leq C_{\rm sb} (C_\cP h_N)^2 \sum_{P \in \cP_N} \int_P |\nabla s_i(x)|^2 dx = C h_N^2 \|\nabla s_i\|^2,
\end{align}
where the Sobolev inequality constant $C_{\rm sb}$ only depends on the dimension, and $C := C_{\rm sb} C_\cP^2$. Recall that, in all this paper, $\|\cdot\|$ stands for the $L^2(X, \mR^n)$ norm, and $|\cdot|$ for the euclidean norm on $\mR^n$, for any integer $n \geq 1$.
The map $\mP_N$ is $1$-Lipschitz, as the orthogonal projection onto the convex set $\mM_N$. 
Hence for any $0 \leq i < T$
\begin{align}
\label{eq:ssClassical}
\|m_i - m_{i+1}\|^2 &\leq \|s_i - s_{i+1}\|^2 \leq \frac 1 T \int_{\frac i T}^{\frac{i+1} T} \|\dot s(t) \|^2 dt.
\end{align}
Summing \eqref{eq:msClassical} and \eqref{eq:ssClassical} over $0 \leq i \leq T$ we obtain
\begin{align*}
\cE(T,N,\lambda)  &\leq T\sum_{0 \leq i < T} \|m_{i+1}-m_i\|^2 + \lambda\sum_{0 \leq i \leq T} \|m_i-s_i\|^2 \\
& \leq \sum_{0 \leq i < T} \int_{\frac i T}^{\frac{i+1} T} \| \dot s(t) \|^2 dt + \lambda \sum_{0 \leq i \leq T} C h_N^2 \|\nabla s_i\|^2 \\
& \leq d^2(s_*,s^*) + C' T h_N^2\lambda,
\end{align*}
where $C' = C \|s\|_{L^\infty([0,1], H^1(X))}$, which concludes the proof.

\begin{figure}
\centering
\includegraphics[height=3cm]{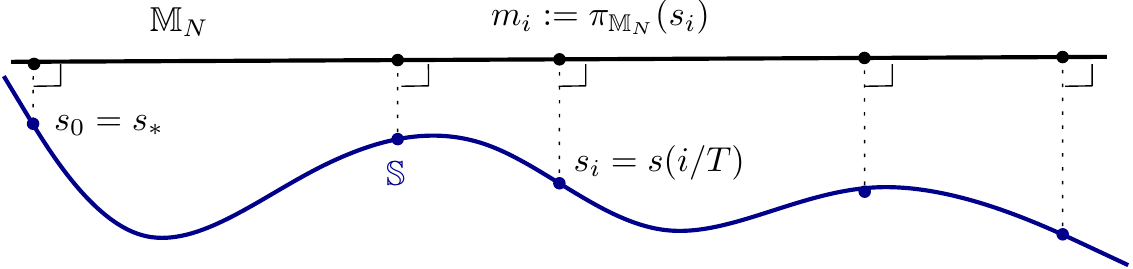}
	\caption{Theorem \ref{th:UpperBound} (classical estimate) is based projecting the measure preserving maps $(s_i)_{i=0}^T \in \mS^{T+1}$ onto the finite dimensional space $\mM_N$, a procedure symmetric the projection of $(m_i)_{i=1}^{T-1} \in \mM_N^{T-1}$ onto $\mS$ involved in the discrete energy optimization \eqref{eqdef:DiscreteOptimization}, see Figure \ref{fig:ProjMToS}.}
	\label{fig:ProjSToM}
\end{figure}

\subsection{Length of a chain of incompressible maps}
\label{subsec:IncompressibleChain}

Proposition \ref{prop:Construction} (Classical construction) immediately follows from Lemma \ref{lem:Length} below, which is general and could be used to approximate geodesics on any manifold $\mS$ embedded in a Hilbert space $\mM$, internally approximated by subspaces $\mM_N$.
It relies on a the following identity, valid for any elements $a,b$ of a Hilbert space, and any $\ve >0$:
\begin{equation}
\label{eq:SquareSumEps}
(1+\ve)^{-1}\|a+b\|^2 \leq  \|a\|^2+ \ve^{-1} \| b\|^2.
\end{equation}
Indeed subtracting the LHS to the RHS of \eqref{eq:SquareSumEps} we obtain $(1+\ve)^{-1} \| \ve^{\frac 1 2}a - \ve^{-\frac 1 2}b \|^2 \geq 0$.

\begin{lemma}
\label{lem:Length}
For any $T \in \mN^*$, any penalization $\lambda >0$, and any $(m,s) \in (\mM\times \mS)^{T+1}$ one has
\begin{equation}
\label{eq:LengthEnergy}
T \sum_{0 \leq i < T} \|s_{i+1} - s_i\|^2 \leq (1+4T /\lambda) \left[ T\sum_{0 \leq i < T} \|m_{i+1}-m_i\|^2 + \lambda\sum_{0 \leq i \leq T} \|m_i-s_i\|^2 \right].
\end{equation}
\end{lemma}

\begin{proof}
Let $0 \leq i < T$. Choosing $a:=s_{i+1}-s_i$, and $b := m_{i+1}-m_i -a$, we obtain
\begin{align*}
(1+\ve)^{-1}\|m_{i+1}-m_i\|^2 &\leq \|s_{i+1}-s_i\|^2 + \ve^{-1} \|(s_i-m_i)-(s_{i+1}-m_{i+1})\|^2\\
&\leq \|s_{i+1}-s_i\|^2 + 2 \ve^{-1} (\|s_i-m_i\|^2+\|s_{i+1}-m_{i+1}\|^2).
\end{align*}
Summing over $0 \leq i< T$ yields
\begin{equation*}
 (1+\ve)^{-1} \sum_{0 \leq i < T} \|m_{i+1}-m_i\|^2 \leq \sum_{0 \leq i < T} \|s_{i+1}-s_i\|^2 + 
2\ve^{-1} \sum_{0 \leq i \leq T} \alpha_i \|s_i-m_i\|^2,
\end{equation*}
 with $\alpha_0=\alpha_T=1$, $\alpha_i = 2$ otherwise. 
Choosing $\ve=4 T/\lambda$ concludes the proof.
\end{proof}

The second point of Proposition \ref{prop:Construction} is based on \eqref{eq:SquareSumEps} as well.
Indeed, let $(m_i)_{i=0}^T$ be minimizers of \eqref{eqdef:DiscreteOptimization}, let $0 \leq i < T$, and let $t =(i+\alpha)/T$ with $0 \leq \alpha\leq 1$. Then for any $\ve>0$ 
\begin{equation*}
	W_2^2(\eval_t \# \mu, \Leb) = \inf_{s \in \mS} \|(1-\alpha) m_i + \alpha m_{i+1} -s \|^2 \leq (1+\ve) \left(  \| \alpha(m_{i+1}-m_i)\|^2 + \ve^{-1} \inf_{s \in \mS} \|m_i-s\|^2\right)
\end{equation*}
Integrating over $t \in [0,1]$, using that either $\alpha \leq 1/2$ or $1-\alpha \leq 1/2$, and choosing $\ve=4T/\lambda$, we obtain as announced
\begin{equation*}
	\int_0^1 W_2^2(\eval_t \# \mu, \Leb) dt \leq \frac {1+\ve} T \sum_{0 \leq i < T} \left(\frac 1 4\|m_{i+1}-m_i\|^2 + \ve^{-1} \inf_{s \in \mS} \|m_i-s\|^2\right) = \frac 1 {4 T^2}\cE'(T,N,\lambda).
\end{equation*}
Finally, the convergence claim for the minimizing chain $(\mu_T)_{T \in \mN}$ results from classical arguments.
(i) The weak-* lower semi-continuity of the energy $\mu \mapsto \int_\Omega \cA(\omega) d \mu(\omega)$ on $\Prob(\Omega)$, which follows from the lower semi-continuity of the action $\cA : \Omega \to \mR_+ \cup \{\infty\}$.
(ii) The weak-* sequential compactness of $\{\mu \in \Prob(\Omega); \, \int_\Omega \cA(\omega) d \mu(\omega) \leq K\}$ for any constant $K$, see \cite{Brenier:1993wd}.
(iii) The weak-* continuity of $\mu \mapsto W_2^2( (\eval_0, \eval_1) \# \mu, (s_*,s^*) \# \Leb)$, a quantity bounded for $\mu_T$ by $\|m_0-s_*\|^2+\|m_T-s^*\|^2 \leq \cE(T,N_T,\lambda_T)/\lambda_T \to 0$ as $T \to \infty$. 
(iv) The weak-* lower semi-continuity of $\mu \mapsto \int_0^1 W_2^2(\eval_t \# \mu, \Leb) \, dt$, which follows from Fatou's lemma and the continuity of $\mu \mapsto W_2^2(\eval_t \# \mu, \Leb)$ for any $t \in [0,1]$. 


\section{Relaxed analysis}
\label{sec:Relaxed}

%

We prove Theorem \ref{th:UpperBound} (Relaxed estimate), using a quantization of the generalized flow minimizing the relaxed geodesic distance \eqref{eqdef:MinimizeRelaxed}. This quantization is a counterpart of the partition $\cP_N$ of the domain $(X, \Leb)$ used for the classical estimate \S \ref{subsec:ClassicalUpperBound}, which amounts to quantize the initial positions of the fluid particles.
Let $\delta_x$ denote the Dirac probability measure concentrated at a point $x$.

\begin{definition}
\label{def:Dimension}
	Let $\mH$ be a metric space, let $\mu$ be a probability measure on $\mH$, and let $\Gamma \subset \mH$. For all $N \geq 1$ denote, with $W_2$ the Wasserstein distance for the quadratic transportation cost
	\begin{align*}
		h_N(\mu) &:= \inf_{\omega \in \mH^N} W_2\biggl(\mu, \ \frac 1 N \sum_{1 \leq i \leq N} \delta_{\omega_i}\biggr), &
		r_N(\Gamma) &:= \inf_{\omega \in \mH^N} \min \biggl\{ r \geq 0; \ \Gamma \subset \bigcup_{1 \leq i \leq N} \overline B(\omega_i, r) \biggr\}.
	\end{align*}
	The quantization dimension of $\mu$, and the box dimension of $\Gamma$, are defined by 
	\begin{align*}
	D_{\rm quant}(\mu) &:= \limsup_{N \to \infty} \frac {\ln N}{-\ln h_N(\mu)},
	&
	D_{\rm box}(\Gamma) &:= \limsup_{N \to \infty} \frac {\ln N}{-\ln r_N(\Gamma)}. 
	\end{align*}
\end{definition}

The decay rate of $h_N$ is directly involved in the announced result Theorem \ref{th:UpperBound}. We estimate it using an elementary result of quantization theory, and refer to \cite{Gersho:1992wy}
 for more details on this rich subject.
Note that the (upper) box dimension $D_{\rm box}$ is a variant of the Haussdorff dimension, in which the set of interest if covered by balls of \emph{equal} radius. Box and Haussdorff dimension coincide for compact manifolds, but differ in general. For instance, all countable sets have Haussdorff dimension zero, whereas one can check that
\begin{align*}
	D_{\rm box} \left( ([0,1] \cap \mQ)^d \right) &= d, & D_{\rm box} \left(\left\{\frac 1 n;\, n \in \mN^*\right\}\right) &= \frac 1 2.
\end{align*}

\begin{proposition} 
\label{prop:Dimension}
Let $\mH$ be a metric space, and let $\mu \in \Prob(\mH)$ be supported on a set $\Gamma$. Then $D_{\rm quant}(\mu) \leq \max \{2,D_{\rm box}(\Gamma) \}$. More precisely for any $D>0$, one has as $N \to \infty$
	\begin{equation}
	\label{eq:DeltaHDecay}
		r_N(\Gamma) = \cO(N^{-\frac 1 D})
		\quad \Rightarrow \quad 
		h_N(\mu) = \cO
		\begin{cases} 
			N^{-\frac 1 D} &\text{ if } D>2,\\
			N^{-\frac 1 2} \sqrt{\ln N}  &\text{ if } D=2, \\ 
			N^{-\frac 1 2} &\text{ if } D<2.
		\end{cases}
	\end{equation}
\end{proposition}

\begin{proof}
	Let $N \in \mN$ be fixed. For each $1 \leq i \leq N$ let $M_i\subset \mH$ be a set of $i$ points such that $\Gamma \subset \cup_{\omega \in M_i} \overline B(\omega, 2 r_i)$, with $r_i := r_i(\Gamma)$. We construct a sequence of points $\omega_i \in \mH$, and an increasing sequence of measures $\rho_i$ supported on $\Gamma$ and of mass $i/N$, inductively starting with $i=N$ and finishing with $i=1$. Initialization: $\rho_N:=\mu$.
	
	Induction: for each $1 \leq i \leq N$, we construct $\omega_i$ and $\rho_{i-1}$ in terms of $\rho_i$. Let indeed $\omega_i \in M_i$ be such that $B_i := \overline B(\omega_i, 2r_i)$ satisfies $\rho_i(B_i) \geq 1/N$. Such a point exists since $|\rho_i| = i/N$, $\#(M_i)=i$, and $\supp(\rho_i) \subset \Gamma$. 
	Then let $\rho_{i-1} := \rho_i - \frac 1 {N\rho_i(B_i)} \rho_i$, so that $\rho_i - \rho_{i-1}$ is a non-negative measure of mass $\frac 1 N$ supported on $\overline B_i$. One has 	
	\begin{equation*}
		h_N(\mu)^2 \leq W_2^2\biggl(\mu, \frac 1 N \sum_{1 \leq i \leq N} \delta_{\omega_i}\biggr) \leq \sum_{1 \leq i \leq N} W_2^2 \biggl(\rho_i - \rho_{i-1}, \frac 1 N\delta_{\omega_i}\biggr) \leq \frac 1 N\sum_{1 \leq i \leq N} (2 r_i)^2.
	\end{equation*}
	The comparison \eqref{eq:DeltaHDecay} of the decay rates of $h_N(\mu)$ and $r_N(\Gamma)$ immediately follows. Finally the comparison of the dimensions follows from \eqref{eq:DeltaHDecay}.
\end{proof}

We now specialize the choice of $\mu$, $\Gamma$ and $\mH$.
Let $\mu \in \Prob(\Omega)$ be a generalized flow minimizing the relaxed geodesic distance \eqref{eqdef:MinimizeRelaxed}. This measure is concentrated on the set $\Gamma$ of paths obeying Newton's second law of motion
\begin{equation*}
	\Gamma := \{\omega \in C^2([0,1],X); \, \forall t \in [0,1], \ \ddot\omega(t) = -\nabla p(t,\omega(t))\},
\end{equation*}
where the pressure gradient $\nabla p : [0,1] \times X \to \mR^d$ is assumed, following the assumptions of Theorem \ref{th:UpperBound}, to have Lipschitz regularity.
We regard $\Gamma$ as embedded in the Hilbert space $\mH := H^1([0,1], \mR^d)$, which plays a natural role in the problem of interest \eqref{eqdef:MinimizeRelaxed} and is equipped with the norm
\begin{equation*}
	\| \omega\|^2_{\mH} := \left|\int_0^1 \omega\right|^2 + \int_0^1 |\dot \omega|^2. 
\end{equation*}
Note that $\mH$ continuously embeds in $C^0(\Omega, \mR^d)$, hence the evaluation maps $\eval_t : \mH \to \mR^d$ are continuous with a common Lipschitz constant denoted $C_{\eval}$.

\begin{lemma}
\label{lem:CauchyLipschitz}
	The set $\Gamma$ is compact. Furthermore the map $\Gamma \to X \times \mR^d : \omega \mapsto (\omega(0), \dot\omega(0))$ is bijective and bi-Lipschitz onto its image. 
\end{lemma}

\begin{proof}
	The result follows from Cauchy-Lipschitz's theorem for ordinary differential equations, and the compactness of $X$. 
\end{proof}

The image of the generalized flow $\mu$ by the map of Lemma \ref{lem:CauchyLipschitz}, namely initial position and speed, is often called a minimal measure \cite{Bernot:2009ed}.
Since there is no ambiguity, we  denote $h_N := h_N(\mu)$. The constants $c,C,C'$ appearing in the estimates below only depend on the dimension $d$.

\begin{corollary}
\label{prop:IntermediateDimension}
One has $h_N =\cO(N^{-\frac 1 {2d}})$  (resp.\ $\cO(N^{-\frac 1 2} \sqrt{\ln N})$ if $d=1$.)
\end{corollary}

\begin{proof}
By Lemma \ref{lem:CauchyLipschitz}, the set $\Gamma$ is in bi-Lipschitz bijection with a compact set $K \subset \mR^{2d}$. Hence $r_N(\Gamma) \leq C r_N(K) \leq C' N^{-\frac 1 {2d}}$, and the upper estimate follows from \eqref{eq:DeltaHDecay}.
\end{proof}

The quantization scale $h_N$ is also bounded below, and is minimal for classical solutions. 

\begin{lemma}
\label{lem:LowerBound}
There exists $c>0$ such that $h_N \geq c N^{-\frac 1 d}$ for all $N>0$. If the generalized flow $\mu$ in fact represents a classical solution $s$ to Euler's equations, and $\nabla \dot s$ is bounded on $[0,1] \times X$, then this lower estimate is sharp: $h_N = \cO(N^{-\frac 1 d})$.
\end{lemma}

\begin{proof}
Since $X$ is a $d$-dimensional domain, there exists $c>0$ such that $W_2(\Leb, \nu_N) \geq c N^{-\frac 1 d}$ for any measure $\nu_N$ supported at $N$ points of $\mR^d$. (Recall that, in this paper, $\Leb$ denotes the Lebesgue measure restricted to the set $X$, and normalized for unit mass.)
The first point follows: for any measure $\mu_N$ supported at $N$ points of $\mH$
\begin{equation*}
	cN^{-\frac 1 d} \leq W_2(\Leb, \eval_0 \# \mu_N) = W_2(\eval_0\# \mu, \eval_0 \# \mu_N)\leq C_{\eval} W_2(\mu, \mu_N) = C_{\eval} h_N.
\end{equation*}
Second point: for each $x \in X$, let $\omega_x : t \mapsto s(t,x)$. Then $\Phi : (X, \Leb) \to (\Gamma, \mu) : x \mapsto \omega_x$ is measure preserving and Lipschitz, with regularity constant denoted $C_\Phi$. 
Let $\nu_N$ be a discrete probability measure, with one Dirac mass of weight $1/N$ in each region of the partition $\cP_N$. Since these regions have diameter $\leq C_\cP N^{-\frac 1 d}$, we conclude that 
\begin{equation*}
h_N \leq W_2(\mu, \Phi\# \nu_N) = W_2(\Phi\# \Leb, \Phi\# \nu_N) \leq C_\Phi W_2(\Leb, \nu_N) \leq C_\Phi C_\cP N^{-\frac 1 d}.
	\qedhere
\end{equation*}
\end{proof}

In the rest of this section, we fix the integer $N$ and allow ourselves a slight abuse of notation: elements $\omega_j, P_j, \rho_j, \ldots$ indexed by $1 \leq j \leq N$ do implicitly depend on $N$, although that second index $\omega_j^N, P_j^N, \rho_j^N, \ldots$ is omitted for readability. 
\begin{lemma}
\label{lem:ExistsOptimal}
	The infimum defining $h_N$ is attained, see Definition \ref{def:Dimension}.
	As a result there exists $(\omega_j)_{j=1}^N\in \mH^N$ and probability measures $(\rho_j)_{j=1}^N$ on $\Gamma$ such that 
	\begin{align}
	\label{eq:Decomp}
		\mu &= \frac 1 N\sum_{1 \leq j \leq N} \rho_j & h_N^2 &= \frac 1 N \sum_{1 \leq j \leq N} \int_{\Gamma} \|\omega-\omega_j\|^2_\mH \, d\rho_j(\omega)
	\end{align}
	Furthermore, $\omega_j$ is the barycenter of $\rho_j$ for each $1 \leq j \leq N$. 
\end{lemma}

\begin{proof}
Let $(\omega_j)_{j=1}^N$ be a candidate quantization, and let $\pi$ be the transport plan associated to $W_2^2(\frac 1 N \sum_{j=1}^N \delta_{\omega_j},\mu)$. Then the measures $\rho_j : A \mapsto N\, \pi(\{x_j\} \times A)$, $1 \leq j \leq N$, are probabilities which average to $\mu$, and the transport cost is the RHS of \eqref{eq:Decomp}. 
The quantization energy, i.e. the squared Wasserstein distance, is  decreased by replacing $\omega_j$ with the barycenter $b_j$ of $\rho_j$, $1 \leq j \leq N$, by the amount $\frac 1 N \sum_{j=1}^N |\omega_j-b_i|^2$. Hence $\omega_j = b_j$ for all $1 \leq j \leq N$ if the quantization is optimal. Note also that the barycenter of $\rho_j$ belongs to $G := \overline {\Hull(\Gamma)}$ by construction.

Since $\Gamma$ is a compact subset of a Hilbert space, the convex hull closure $G$ is also compact, for the strong topology induced by $\|\cdot\|_\mH$. 
The quantization energy $(\omega_j)_{j=1}^N \mapsto W_2^2(\frac 1 N \sum_{j=1}^N \delta_{\omega_j},\mu)$ attains its minimum on $G^N$ by compactness, and by the previous argument it is the global minimum on $\mH^N$.
\end{proof}
	
Let $\mu_N$ denote the equidistributed probability on the set $\{\omega_j;\ 1 \leq j \leq N\}$ of Lemma \ref{lem:ExistsOptimal}.

\begin{lemma}
\label{lem:Indexation}
The regions of the partition $\cP_N$ of $\Omega$ can be indexed as $(P_j)_{j=1}^N$ in such way that 
\begin{equation}
\label{eq:Indexation}
C h_N^2 \geq  \sum_{1 \leq j \leq N} \int_{P_j} |\omega_j(0) - x|^2 dx.
\end{equation}
\end{lemma}
\begin{proof}
Let $B_N \subset \Omega$ collect the barycenters of the partition $\cP_N$, and let $\nu_N$ denote the equidistributed probability on $B_N$. 
	One has 
	\begin{align*}
		W_2(\nu_N, \Leb) &\leq C_\cP N^{-\frac 1 d}, & W_2(\Leb, \eval_0 \# \mu_N) &\leq C_{\eval} W_2(\mu, \mu_N) = C_{\eval} h_N.
	\end{align*}
	Thus $W_2(\nu_N, \eval_0 \# \mu_N) \leq C_1 h_N$ by Lemma \ref{lem:LowerBound}. This optimal transport problem between the discrete measures $\nu_N$ and $\eval_0 \# \mu_N$ determines an optimal assignment $\Gamma_N \to B_N$, represented by the indexation $(b_j)_{j=1}^N$ of $\Gamma_N$ and $B_N$. Denoting by $P_j \in \cP_N$ the region of which $b_j$ is the barycenter we conclude that
	\begin{equation*}
		\sum_{1 \leq j \leq N} \int_{P_j} |\omega_j(0)-x|^2 dx = \sum_{1 \leq j \leq N} \int_{P_j} |b_j-x|^2 dx + W_2^2(\nu_N, \eval_0 \# \mu_N) \leq C h_N^2.
		\qedhere
	\end{equation*}
\end{proof}

For each $0 \leq i \leq T$, let $m_i \in \mN$ be the piecewise constant map on the partition $\cP_N$ defined by 
\begin{equation*}
\forall 1 \leq j \leq N, \ \forall x \in P_j, \ m_i(x) := \omega_j(i/T).
\end{equation*}

\paragraph{Bound on the energy terms $\|m_{i+1}-m_i\|$.}
Using Cauchy-Schwartz's inequality we obtain
\begin{align*}
&T \sum_{0 \leq i <T} \|m_{i+1}-m_i\|^2 = \frac 1 N \sum_{1 \leq j \leq N} T \sum_{0 \leq i < T}  \left|\omega_j\biggl(\frac{i+1} T\biggr)-\omega_j\biggl(\frac i T\biggr)\right|^2 
 \leq \frac 1 N\sum_{1 \leq j \leq N} \int_0^1 | \dot \omega_j(t)|^2 dt \\
 & = \frac 1 N \sum_{1 \leq j \leq N} \int_0^1 \left| \int_\Gamma \dot \omega \, d \rho_j(\omega)\right|^2
 \leq \frac 1 N \sum_{1 \leq j \leq N} \int_0^1 \int_\Gamma |\dot \omega|^2 \, d \rho_j(\omega) \, dt
 = d^2(s_*,s^*).
\end{align*}
\paragraph{Distance to incompressible maps.}
For any $1 \leq i \leq T$, with $t := i/T$, one has
\begin{equation*}
\inf_{s \in \mS} \|m_i - s\| = W_2(\Leb, \eval_t \# \mu_N) = W_2(\eval_t \# \mu, \eval_t \# \mu_N) \leq C_{\eval} W_2(\mu, \mu_N) = C_{\eval} h_N.
\end{equation*}

\paragraph{Boundary conditions.}
We make the assumption that $s_* = \Id$, up to a minor modification of Lemma \ref{lem:Indexation} (replace $x$ with $s_*(x)$ in \eqref{eq:Indexation}). Lemma \ref{lem:Indexation} then precisely states that $\|m_0 - s_*\|^2 \leq C h_N^2$, and the generalized boundary condition of \eqref{eqdef:MinimizeRelaxed} states that $\omega(1)=s^*(\omega(0))$ for $\mu$-almost every $\omega \in \Gamma$. Denoting by $C_0$ the Lipschitz regularity constant of $s^*$ we obtain for any $1 \leq j \leq N$ and $x \in X$
\begin{align*}
&|\omega_j(1) - s^*(x)|^2 = \biggl| \int_\Gamma s^*(\omega(0)) - s^*(x)\, d \rho_j(\omega)\biggr|^2 \leq \int_\Gamma |s^*(\omega(0))-s^*(x)|^2 \, d \rho_j(\omega) \\
& \leq C_0^2 \int_\Gamma |\omega(0)-x|^2 \, d \rho_j(\omega) 
 \leq 2 C_0^2 \left( \int_\Gamma |\omega_j(0)-\omega(0)|^2 \, d \rho_j(\omega) + |\omega_j(0)-x|^2 \right).
\end{align*}
Therefore 
\begin{align*}
	\| m_T-s^*\|^2 &= \sum_{1 \leq j \leq N} \int_{P_j} |\omega_j(1) - s^*(x)|^2 \, dx \\
	& \leq 2 C_0^2 \sum_{1 \leq j \leq N} \left(\frac 1 N\int_\Gamma |\omega_j(0)-\omega(0)|^2\, d \rho_j(\omega) + \int_{P_j} | \omega_j(0) - x|^2 \, dx\right)\\
	& \leq 2C_0^2 (C_{\eval}^2 W_2^2 (\mu_N, \mu) + \| m_0-s_*\|^2)  \leq C h_N^2.
\end{align*}
\paragraph{Summation and final estimate.}
The value $\cE(T,N,\lambda)$ of the minimum \eqref{eqdef:DiscreteOptimization} is 
\begin{equation*}
T \sum_{0 \leq i < T} \|m_{i+1}-m_i\|^2+\lambda \left(\|m_0-s_*\|^2+\|m_T-s^*\|^2+\sum_{1 \leq i < T} \inf_{s \in \mS} \|m_i-s\|^2\right) \leq d^2(s_*,s^*)+\cO(T h_N^2 \lambda).
\end{equation*}

\section{Numerical experiments}
\label{sec:Numerics}

\subsection{Minimization algorithm and choice of penalization}
We rely on a quasi-Newton method to compute a (local) minimum of the
discretized problem \eqref{eqdef:DiscreteOptimization}. This means
that we need to compute the value of the functional
\begin{equation}
  m \in \mM_N^{T+1} \mapsto T \sum_{0 \leq i < T} \|m_{i+1}-m_i\|^2 +
  \lambda\biggl(\|m_0-s_*\|^2 + \|m_T-s^*\|^2 + \sum_{1 \leq i < T}
  d^2_\mS(m_i) \biggr).
\label{eq:DiscreteEnergy}
\end{equation} and its gradient, where $d^2_\mS(m) = \inf_{s \in
  \mS}\|m-s\|^2$. The only difficulty is to evaluate the squared
distance $d^2_\mS$ to the set of measure-preserving vector fields and its gradient.  As
explained in the introduction, Brenier's Polar Factorization Theorem
implies that for any vector valued function $m\in\mM$,
$$ d^2_\mS(m) = W^2_2(m\#\Leb, \Leb).$$
When $m$ belongs to $\mM_N$, the measure $m\#\Leb$ is finitely
supported, and the computation of the Wasserstein distance can be
performed using a semi-discrete optimal transport solver \cite{Aurenhammer:1998ie,Merigot:2011js,Levy:2014un}. The
next proposition gives an explicit formulation for the gradient in
term of the optimal transport plan.  Recall that $\mM_N$ is the set of
piecewise constant functions on the tessellation $\cP_N := (P_j)_{1\leq
  j\leq N}$ of $X$. The diagonal $\mD_N$ in $\mM_N$ is the set of
functions $m$ in $\mM_N$ such that $m(P_j) = m(P_k)$ for some $j\neq
k$. The set $\mM_N\setminus \mD_N$ is a dense open set in $\mM_N$.

\begin{proposition} The functional $d^2_{\mS}$ is differentiable
  almost everywhere on $\mM_N$ and continuously differentiable on
  $\mM_N\setminus \mD_N$. The gradient of $d^2_{\mS}$ at $m \in \mM_N
  \setminus \mD_N$ is explicit: with $x_j = m(P_j)$,
\begin{equation}
  \left.\nabla d^2_{\mS}(m)\right|_{P_j} = 
  2 (x_j - \bary(T^{-1}(x_j)))
\end{equation}
where $T:X\to m(X)$ is the piecewise constant optimal transport map
between $\Leb$ and the finitely supported measure $m\#\Leb$ and
$\bary(S) = \int_S x dx / \Leb(S)$ is the isobarycenter of $S$
\end{proposition}

\begin{proof} The functional $\cF := d^2_{\mS}-\|\cdot\|^2$ is concave
  as an infimum of linear functions:
$$\cF(m) = d^2_{\mS}(m) - \|m\|^2
= \inf_{s\in \mS} \|m -s\|^2 - \|m\|^2 = \inf_{s\in \mS}
\left[-2\<m|s\> + \|s\|^2 \right],$$ where $\<\cdot, \cdot\>$ denotes
the $L^2(X)$ scalar product.  This implies in particular that $\cF$
and therefore $d^2_{\mS}$ is differentiable almost everywhere on
$\mM_N$. Given $m$ in $\mM_N\setminus \mD_N$, define $x_j = m(P_j)$
and let $T : X \to \mR^d$ be the optimal transport plan from $\Leb$ to
$m\#\Leb = \frac{1}{N} \sum_{j=1}^N \delta_{x_j}$.  The transport plan
is indeed always representable by a function when the source measure
is absolutely continuous with respect to the Lebesgue
measure. 
Let $V_j = T^{-1}(x_j)$ be the partition of $X$ induced by this
transport plan. Then
\begin{align*}
\cF(m) 
= W_2^2(m\#\Leb,\Leb) - \|m\|^2 &= \sum_{j=1}^N \int_{V_j} \|x_j - x\|^2 - \|x_j\|^2 d x \\
&= \<m|G(m)\> + \sum_{j=1}^N \int_{V_j} \|x\|^2 dx
\end{align*} 
where $G(m) \in \mM_N$ is the piecewise constant function on $X$ given
by $\left.G(m)\right|_{P_j} = -2 \bary(V_j)$. 
For any $m'$ in $\mM_N$ and $x'_j = m'(V_j)$, one has
\begin{align*}
\cF(m')
= W_2^2(m'\#\Leb,\Leb)-\|m'\|^2
 &\leq \sum_{j=1}^N \int_{V_j} \|x_j' - x\|^2  - \|x_j'\|^2d x  \\
&= \cF(m) + \<m'-m|G(m)\>
\end{align*}
This shows that $G(m)$ belongs to the superdifferential to $\cF$ at
$m$. In addition, by the continuity of optimal transport plans, the map $m\in \mM_N \setminus \mD_N \mapsto G(m)$ is
continuous.
To summarize, on the open domain $\mM_N\setminus \mD_N$
the concave function $\cF$ possesses a continuous selection of
supergradient. This implies that $\cF$ is of class $\cC^1$ on this
domain, with $\nabla F(m) = G$, and the result follows.
\end{proof}

\paragraph{Construction of the initial solution} Since the discrete
energy \eqref{eq:DiscreteEnergy} is non-convex, the construction of
the inital guess is important. We follow a time-refinement strategy
already used by Brenier \cite{Brenier:2008ho} to construct a good
initial guess. Assuming that we have already a local minimizer for
$T_{k} = 2^k+1$, we use linear interpolation to construct an initial
guess for $T_{k+1} = 2^{k+1} +1$. The optimization is then performed
from this inital guess, using a quasi-Newton algorithm for the energy
\eqref{eq:DiscreteEnergy}.

\paragraph{Choice of the penalization parameter}
The optimal choice of 
$\lambda$ in \eqref{eq:DiscreteEnergy} depends on the
quantization dimension $D = D_{\rm quant}(\mu)$ of the generalized solution $\mu\in \Prob(\Omega)$ that one expects to
recover: namely $\lambda_N = N^{-\frac 1 D}$, see the remark after \eqref{eq:ChoiceLambda}.
We call $D$ the flow dimension, and regard it as as the intrinsic dimensionality
of the problem which determines its computational difficulty. For a
classical solution, this dimension agrees with the ambient dimension
i.e. $D=d$, while for a non-deterministic solution the quantization
dimension can be up to $2d$. Intermediate dimensions $d< D < 2d$ are
also common \cite{Brenier:1985bz}. 
In our numerical experiments we set $\lambda_N = N^\frac 1 3$, a decision justified a-posteriori by the numerical estimation of the
quantization dimension of the computed solution, see
Figure~\ref{fig:Boxcount}.

Note that the numerical error in \eqref{eq:ChoiceLambda} is governed (for a fixed number $T$ of time steps) by the quantity $\lambda^{-1} + h_N^2 \lambda$, and that $h_N = \cO(N^{-\frac 1 {2d}})$ 
under the assumptions of Theorem \ref{th:UpperBound}. The choice $\lambda_N = N^{\frac 1 \alpha}$ thus yields a convergent scheme whenever $\alpha > d$, although convergence rates are improved if $\alpha$ is close to the flow dimension $D$, so that $\lambda_N \approx N^\frac 1 D \approx h_N^{-1}$.

\subsection{Visualization of generalized solution}
\label{subsec:Visualization}
The main interest of numerical experimentation is to visualize
generalized solutions to Euler's equation, or equivalently generalized
geodesics between two measure-preserving diffeomorphisms $s_*,s^*$ in
$\mS$.

\subsubsection{Gradient of the pressure}
\label{subsubsec:GradientPressure}
Consider a minimizer of the discretized energy
\eqref{eq:DiscreteEnergy}. Given $i \in \{ 1,\hdots, T-1\}$, $m_i$
minimizes over $\mM_N$ the functional $m \mapsto
T(\|m-m_{i-1}\|^2+\|m_{i+1}-m\|^2) + \lambda d_\cS^2(m)$.  This gives
\begin{equation}
\label{eq:Acceleration}
T^2(m_{i-1}-2 m_i+m_{i+1}) = T \lambda \nabla d_\cS^2(m_i).
\end{equation}
This equation is a discretized counterpart of the rule that the
acceleration of a geodesic on an embedded manifold, is normal to that
manifold (here $\mS$ plays the role of the manifold, embedded in
$\mM$, which is internally approximated by the linear space
$\mM_N$). The second order difference $T^2 (m_{i-1}-2 m_i+m_{i+1})$
approximates a second derivative in time. Comparing
~\eqref{eq:Acceleration} to \eqref{eq:PathDynamics}, we see that the
right hand-side of \eqref{eq:Acceleration} can be used as an
estimation of (minus) the pressure of the gradient.

\subsubsection{Geometric data analysis} As in the proof of
Theorem~\ref{th:UpperBound}, the discrete minimizer of \eqref{eq:DiscreteEnergy} can converted to a
collection of $N$ piecewise-linear curves
$\{\omega_1,\hdots,\omega_N\} = \Gamma_N$. We recall that the domain $X$ is
partitioned into $N$ subdomains $(P_j)_{1\leq j\leq N}$ with equal
area and we let $\omega_j(i/T) \in \mR^d$ be the point corresponding to the
restriction of $m_i$ to the subdomain
$P_j$, for each $0 \leq i \leq T$. Figure~\ref{fig:SolutionToCurves} illustrates this
construction. We regard $\Gamma_N$ as embedded in the Hilbert space $\mH := H^1([0,1], \R^2)$ which plays a natural role in the problem of interest, as in \S \ref{sec:Relaxed}, and apply techniques from the field of geometric
data analysis.

\label{subsubsec:GDA}
\paragraph{Clustering} In order to better visualize the solution, we use the $k$-means algorithm to divide the set $\Gamma_N$ into $k$. A distinct particle color is attached to each cluster, see for instance
Figure~\ref{fig:Square-Reconstructed-Clusters}.  The $k$-means
algorithm consists in finding a local minimizer of the optimal
quantization problem
\begin{equation}
	\label{eq:Box2}
\min_{\ell_1,\hdots\ell_k \in \mH} \frac{1}{N} \sum_{\omega\in \Gamma_N} \min_{1\leq i\leq k} 
\|\omega - \ell_i\|_\mH^2 	
\end{equation}
using a simple fixed point algorithm, and
to divide $\Gamma_N$ into clusters $(C_i)_{1\leq i\leq k}$ with
$$C_i = \left\{ \omega \in \Gamma_N; \|\omega - \ell_i\|_\mH = \arg\min_{1\leq j\leq k} \|\omega - \ell_i\|_\mH \right\}.$$
Note that $l_1, \cdots, l_k$ automatically belong to ${\rm Span}(\Gamma_N)$, hence to the $d(T+1)$-dimensional linear subspace of $\mH$ consisting of piecewise linear paths with nodes $\omega(t) \in \mR^d$ at times $t=i/T$, $0 \leq i \leq T$. This makes \eqref{eq:Box2} tractable.

\paragraph{Box dimension } 
A natural objective is to estimate the quantization dimension $D_{\rm
  quant}(\mu)$ of the generalized flow $\mu \in \Prob(\Omega)$
minimizing the relaxed problem \eqref{eqdef:MinimizeRelaxed}. The
probability measure $\mu_N$ equidistributed on the set $\Gamma_N$
approximates $\mu$, see Proposition \ref{prop:Construction}, hence we
can expect the set $\Gamma_N$ to also approximate $\supp(\mu)$.  
The quantization dimension $D_{\rm quant}(\mu)$ is difficult to
estimate, but by Proposition~\ref{prop:Dimension} it admits the
simpler upper bound $D_{\rm box}(\supp(\mu))$.  We estimate the latter
by applying the furthest point sampling algorithm to the finite
metric space $\Gamma_N$, which defines an ordering on the elements of
$\Gamma_N$ as follows: let $\gamma_1$ be an arbitrary point of
$\Gamma_N$ and define by induction
\begin{equation}
\gamma_{i+1} := \arg\max_{\gamma \in \Gamma_N} d(\gamma, \{\gamma_1,\hdots,\gamma_i\}) 
\label{eq:FNS}
\end{equation}
As in Definition~\ref{def:Dimension}, denote by $r_i = r_i(\Gamma_N)$
is the smallest $r\geq 0$ such that $\Gamma_N$ can be covered by $i$
balls of radius $r$.  For $1 \ll i \ll N$, the ratio
$\log(i)/\log(1/r_i(\Gamma_N))$ is expected to approximate $\log(i)/
\log(1/r_i(\supp(\mu)))$ and thus the desired $D_{\rm box}(\supp
\mu)$. 

\begin{lemma} Let $\ve_i := \max_{\gamma\in \Gamma} d(\gamma,
  \{\gamma_1,\hdots,\gamma_i\})$, where $\gamma_i$ is defined as in
  \eqref{eq:FNS}. Then,
$$
\left(1-\frac{\log(2)}{\log(1/\ve_i)}\right)
\frac{\log(i)}{\log(1/\ve_i)} 
\leq \frac{\log(i)}{\log(1/r_i)}
\leq \frac{\log(i)}{\log(1/\ve_i)}
$$
\end{lemma}
\begin{proof}
  By construction, $r_i \leq \ve_i$. Moreover, the balls
  centered at the points $\gamma_1,\hdots,\gamma_i$ and with radius $\ve_i/2$
  are disjoint, so that $r_i \geq \ve_i/2$.
\end{proof}

\subsection{Test cases and numerical results} 

Our two testcases are constructed from two stationary solutions to
Euler's equation in $2$D. Let $s:\mR_+ \to \mS$ be a classical solution to Euler equation in Lagrangian coordinates \eqref{eq:PathDynamics}, starting from the identity map. We 
solve the discretized version \eqref{eqdef:DiscreteOptimization} of the minimization problem
\eqref{eqdef:MinimizeSDiff}-\eqref{eqdef:MinimizeRelaxed}, with $s_* = s(0)= \Id$ and $s^* := s(t_{\mathrm{max}})$, where $t_{\mathrm{max}} > 0$.
For small values of $t_{\mathrm{max}}$ the solution to this boundary problem is simply the original classical flow $s$, but for larger values a completely different generalized flow is obtained. 
In this case the geodesic $s$ in the space of the measure preserving diffeomorphisms is no longer the unique shortest path between its boundary values $s_*$ and $s^*$.
The first classical behavior is guaranteed if the pressure hessian satisfies 
\begin{equation}
\label{eq:PressureClassicalRescaled}
\nabla^2 p \prec (\pi/t_{\mathrm{max}})^2 \Id	
\end{equation}
uniformly on $[0,t_{\rm max}] \times X$, see \eqref{eq:PressureClassical} and \cite{Brenier:1985bz}. 
In all the numerical experiments, the number of points is set to
$N=10\,000$ and the number of timesteps is $T=2^4+1=17$.

\subsubsection{Rotation of the disk}
On the unit disk $D = \{(x_1,x_2) \in \R^2; \, x_1^2 + x_2^2 \leq 1 \}$,
the simplest stationary solution to Euler's
equation~\eqref{eqdef:Euler} is given by a time-independent pressure
field and speed:
\begin{align*}
p(x_1,x_2) &= \frac{1}{2}(x_1^2 + x_2^2), &
v(x_1,x_2) &= (-x_2,x_1).
\end{align*}
The corresponding Lagrangian flow $s(t)$ is simply the rotation of angle $t$.
The largest eigenvalue of $\nabla^2 p$ is $1$ at every point in
$D$. Hence by \eqref{eq:PressureClassicalRescaled} the flow of rotations
is the unique minimizer to both the variational formulation
\eqref{eqdef:MinimizeSDiff} and its relaxation \eqref{eqdef:MinimizeRelaxed} with boundary values $s_*=s(0)=\Id$ and $s^*=s(t_{\rm max})$, when $t_{\rm max} < \pi$. Uniqueness is lost at the critical time $t_{\rm max}=\pi$ which corresponds to a rotation of  angle $\pi$, so that
the final diffeomorphism becomes $s_* = s(\pi) = -\Id$. In this situation, the minimization problem \eqref{eqdef:MinimizeSDiff} has
two classical solutions, namely the clockwise and counterclockwise
rotations. The relaxation \eqref{eqdef:MinimizeRelaxed} has uncountably many generalized solutions such as, by linearity, superpositions of these two rotations.

Another explicit example of generalized solution was discovered by 
Brenier \cite{Brenier:1985bz}: given a point $x\in D$ and a speed $v$, denote by
$\omega_{x,v}$ the curve $\omega_{x,v}(t) = x\cos(t) +
v\sin(t)$, $t \in [0,1]$. Then, Brenier's solution is obtained as the pushforward by
the map $(x,v) \mapsto \omega_{x,v} \in \Omega$ of the measure on
$D\times \R^2$ defined by
$$ \mu(dx, dv) = \frac 1 \pi \cH^2(dx) \otimes \frac{1}{2\pi\sqrt{1-\vert x\vert^2}}
\left. \cH^1\right\vert_{\{|v| = \sqrt{1- |x|^2}\}}(dv),$$ where
$\cH^k$ denotes the $k$-dimensional Hausdorff measure.  In particular,
the quantization dimension of the solution is $3=2+1$. We refer to \cite{Bernot:2009ed} for more examples of optimal flows, and construct four dimensional one. Let $\mu_r$ be defined by combining (i) a classical rotation on the annulus $D \sm D(r)$, with $D(r) = \{x \in \mR^2; \, |x| \leq r\}$ and (ii) Brenier's solution rescaled by a factor $r$ on the disc $D(r)$. Then $\mu_r$ is an optimal generalized flow of quantization dimension $3$, whereas the averaged flow $\int_0^1 \mu_r dr$ is also optimal by linearity, and has quantization dimension $4$.

\paragraph{Numerical results} The numerical solutions computed by our
algorithm  for the critical time $t_{\rm max}=\pi$ are highly non-deterministic. To see this, we select a small
neighborhood around several points in the unit disk $D$ and look at
the trajectories emanating from this small neighborhood. As shown in
Figure~\ref{fig:Reconstructed-Trajectories}, we can see that the
trajectories emanating from each neighborhood fill up the disk. 
In addition, each indivual trajectory looks like an
ellipse. Second, we estimate the box dimension of the support of the
numerical solution (as explained in \S\ref{subsec:Visualization}). The
estimated dimension is slightly above $3$. 

\subsubsection{Beltrami flow on the square} On the unit square $S =
[-1/2,1/2]^2$, we consider the Beltrami flow constructed from the
time-independent pressure and speed:
\begin{align*}
p(x_1,x_2) &= \frac{1}{2}(\sin(\pi x_1)^2 + \sin(\pi x_2)^2)\\
v(x_1,x_2) &= (-\cos(\pi x_1) \sin(\pi x_2), \sin(\pi x_1) \cos(\pi x_2)) 
\end{align*}
The maximum eigenvalue of $\nabla^2 p$ is $\pi^2$, and
\cite{Brenier:1985bz} implies that the associated flow is minimizing
between $s_*=s(0)=\Id$ and $s^*=s(t_{\rm max})$ for $t_{\rm max} \leq
1$. Because of the lack of symmetry, generalized solutions constructed
from this flow are less understood than in the disk case.

\paragraph{Numerical results} Our numerical results suggest the
following observations. First, as shown in
Figure~\ref{fig:Square-Reconstructed}, the computed solutions with boundary values $s_*=\Id$ and $s^*=s(t_{\rm max})$
approximate the classical flow if $t_{\rm max} < 1$, and are non-deterministic generalized flows if $t_{\mathrm{max}}\geq 1$. This suggests the
sharpness of the bound given by \cite{Brenier:1985bz}. Interestingly, even for $t>1$, the numerical solutions
seem to remain deterministic in a neighborhood of the boundary of the
cube. This can be seen more clearly in
Figure~\ref{fig:Square-Reconstructed-Clusters}, where the particles
have been divided into clusters using the $k$-means algorithm (see
\S\ref{subsubsec:GDA}).

The pressure gradient is estimated as in
\S\ref{subsubsec:GradientPressure} and is displayed in
Figure~\ref{fig:Square-Reconstructed-Pressure}. These pictures seem
to indicate a loss of regularity of the pressure near the initial and
final times. This corroborates the result of \cite{Ambrosio:2007jk}
according to which the pressure belongs to $L^2_{\rm loc}(\,]0,T[,\,
\BV(X))$.

Figure~\ref{fig:Reconstructed-Trajectories} suggests that the even for
$t_{\mathrm{max}}=1.5$, the reconstructed solution for the Beltrami
flow are more deterministic than the solution to the disk inversion.
We estimate the box dimension of the support of the solution using the
method explained in \S\ref{subsubsec:GDA}. The result are displayed in
Figure~\ref{fig:Boxcount}. The estimated dimension is $D=2$ for the
deterministic solution ($t_{\mathrm{max}} = 0.9$) but it increases as the
maximum time (and therefore the amount of non-determinism)
increases. Finally, we note that the estimated dimensions for
$t_{\mathrm{max}} \in \{1.1, 1.3, 1.5\}$ seem to be strictly between $2$
and $3$, suggesting a fractal structure for the support of the
solution. This would need to be confirmed by a mathematical study.

\paragraph{Software.} The software developed for generating the
results presented in this article is publicly available at
\url{https://github.com/mrgt/EulerSemidiscrete}

\paragraph{Acknowledgement}
The authors thank Y. Brenier for constructive discussions and introducing them to the topic of Euler equations of inviscid incompressible fluids.

\begin{figure}
\centering
\resizebox{\textwidth}{!} 
{
\setlength\tabcolsep{2 pt}
\begin{tabular}{@{}ccccc@{}}
  \subfloat[$t=0.0$]
  {\includegraphics[width=.2\textwidth]{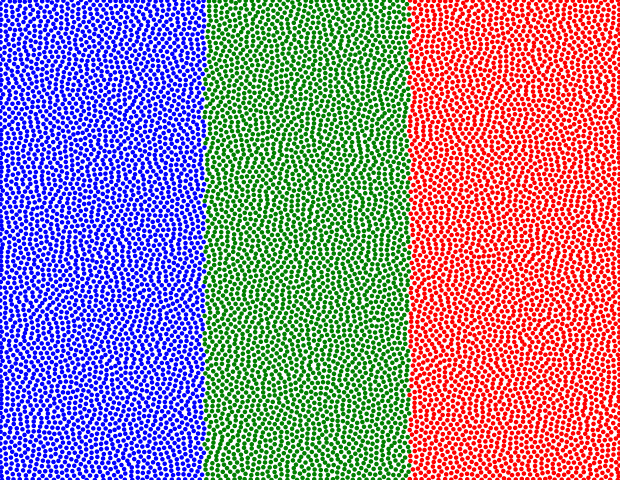}}&
  \subfloat[$t=0.95$]
  {\includegraphics[width=.2\textwidth]{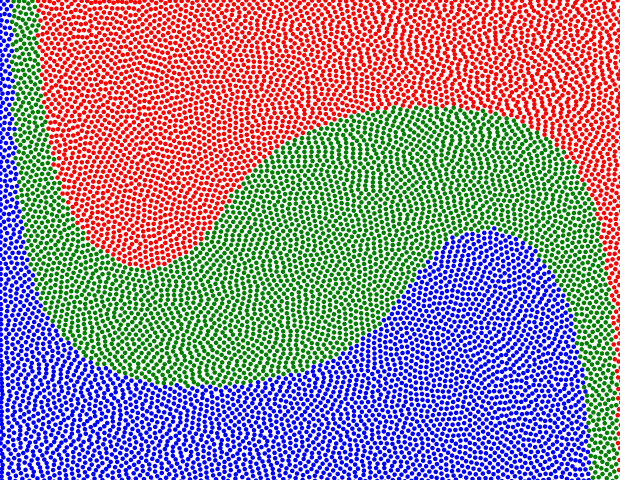}}&
  \subfloat[$t=1.1$]
  {\includegraphics[width=.2\textwidth]{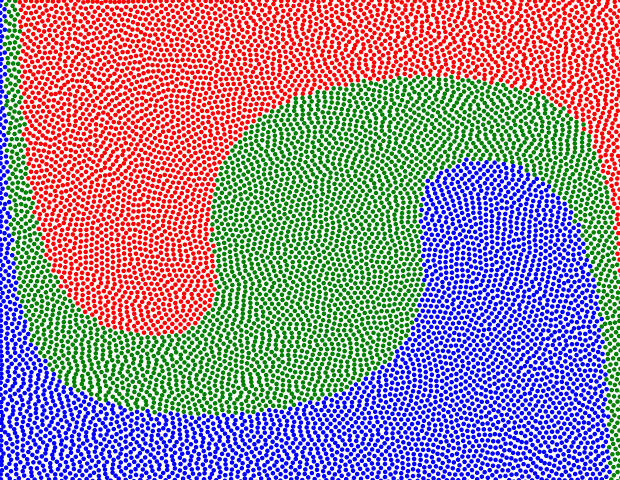}}&
  \subfloat[$t=1.3$]
  {\includegraphics[width=.2\textwidth]{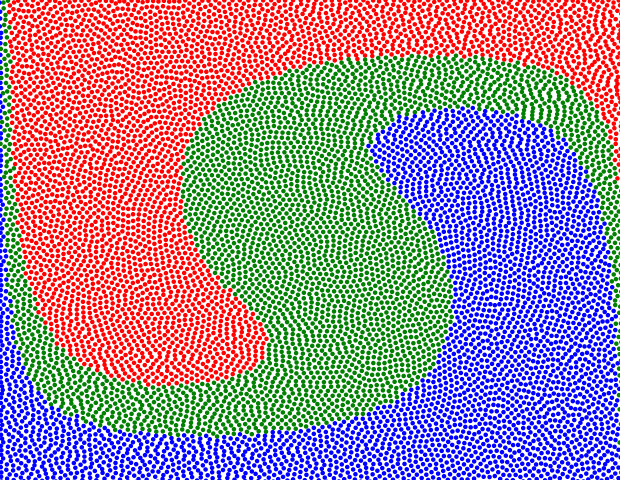}}&
  \subfloat[$t=1.5$]
  {\includegraphics[width=.2\textwidth]{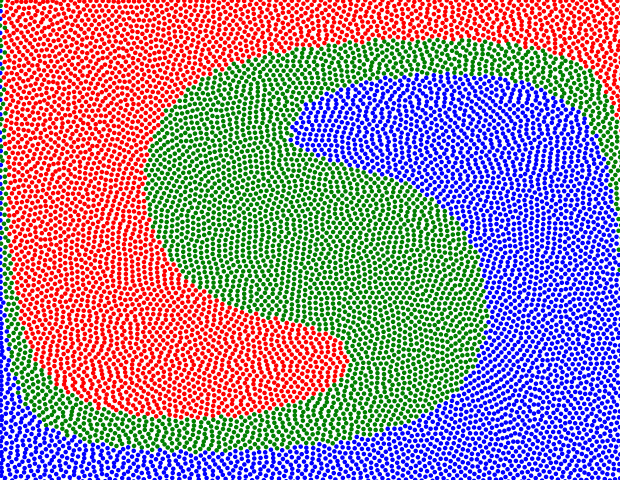}}\\
\hline
  \subfloat[$t=0.0$]
  {\includegraphics[width=.2\textwidth]{Results/Square-Reconstructed-Tmax=0.9/00.png}}&
  \subfloat[$t=0.25*t_{\mathrm{max}}$] 
  {\includegraphics[width=.2\textwidth]{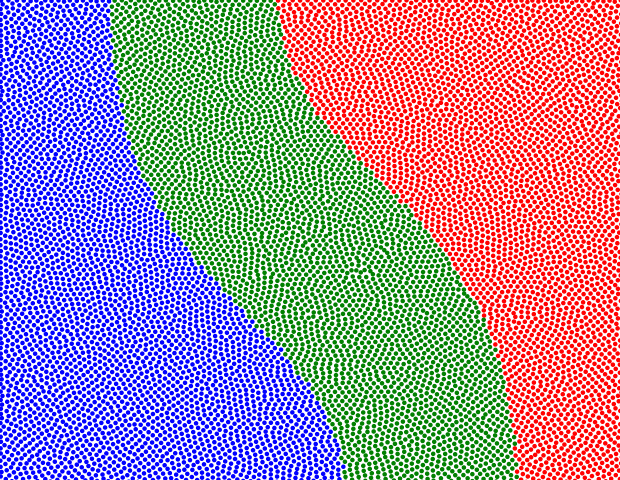}}&
  \subfloat[$t=0.5*t_{\mathrm{max}}$]
  {\includegraphics[width=.2\textwidth]{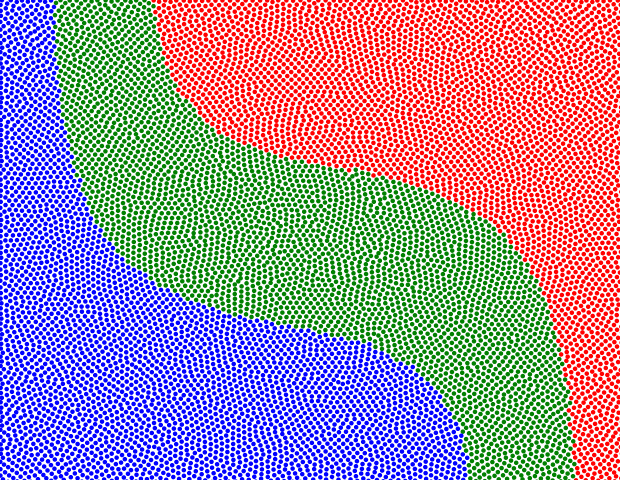}}&
  \subfloat[$t=0.75*t_{\mathrm{max}}$]
  {\includegraphics[width=.2\textwidth]{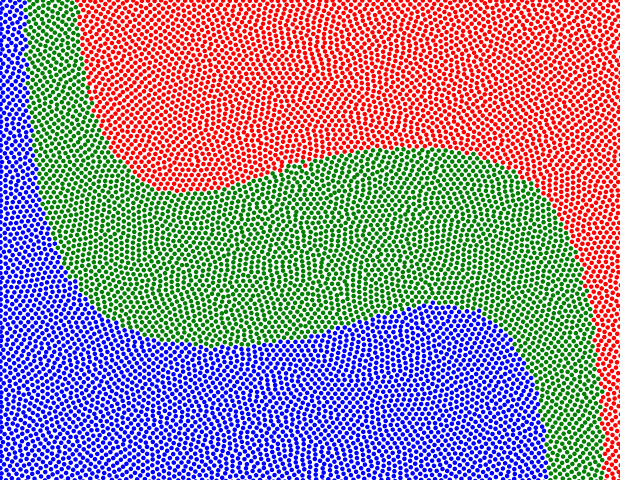}}&
  \subfloat[$t=t_{\mathrm{max}}=0.9$]
  {\includegraphics[width=.2\textwidth]{Results/Square-Reconstructed-Tmax=0.9/16.png}}\\
  \subfloat[$t=0.0$]
  {\includegraphics[width=.2\textwidth]{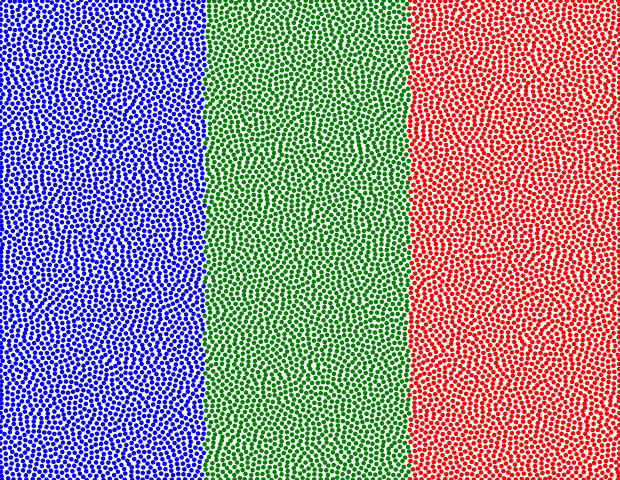}}&
  \subfloat[$t=0.25*t_{\mathrm{max}}$] 
  {\includegraphics[width=.2\textwidth]{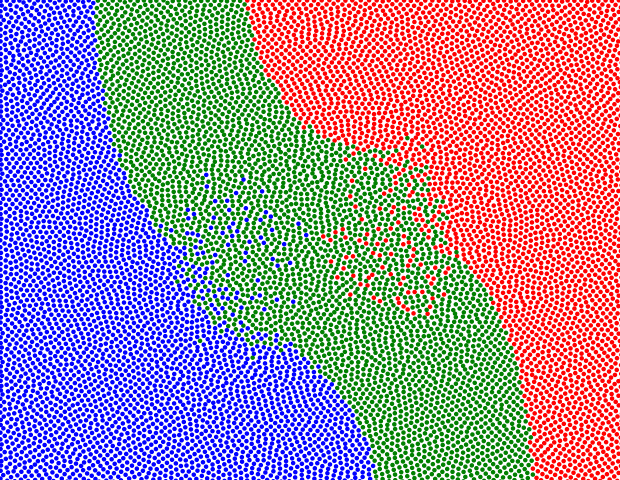}}&
  \subfloat[$t=0.5*t_{\mathrm{max}}$]
  {\includegraphics[width=.2\textwidth]{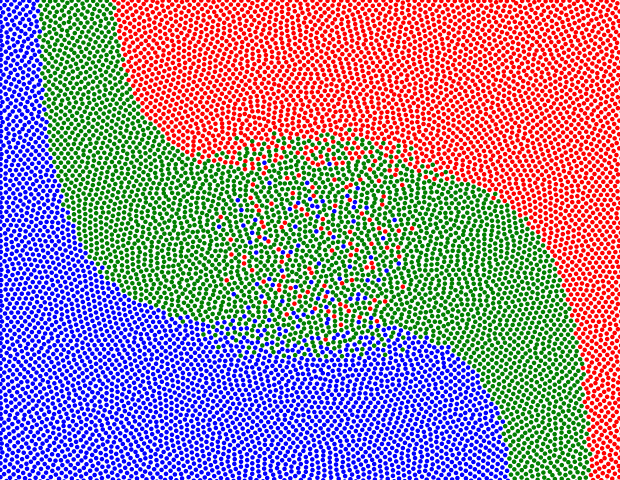}}&
  \subfloat[$t=0.75*t_{\mathrm{max}}$]
  {\includegraphics[width=.2\textwidth]{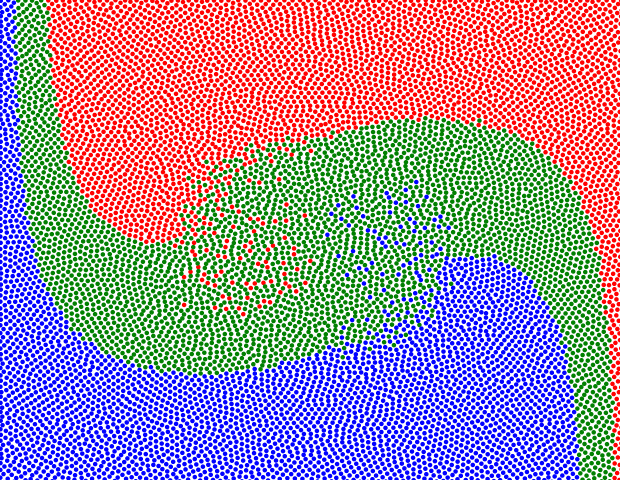}}&
  \subfloat[$t=t_{\mathrm{max}}=1.1$]
  {\includegraphics[width=.2\textwidth]{Results/Square-Reconstructed-Tmax=1.1/16.png}}\\
  \subfloat[$t=0.0$]
  {\includegraphics[width=.2\textwidth]{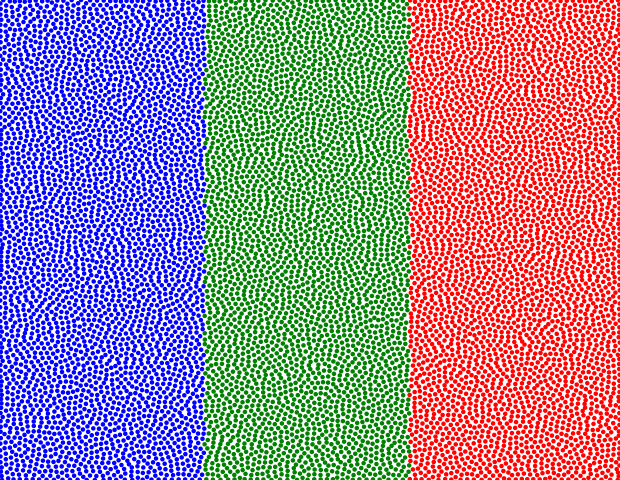}}&
  \subfloat[$t=0.25*t_{\mathrm{max}}$] 
  {\includegraphics[width=.2\textwidth]{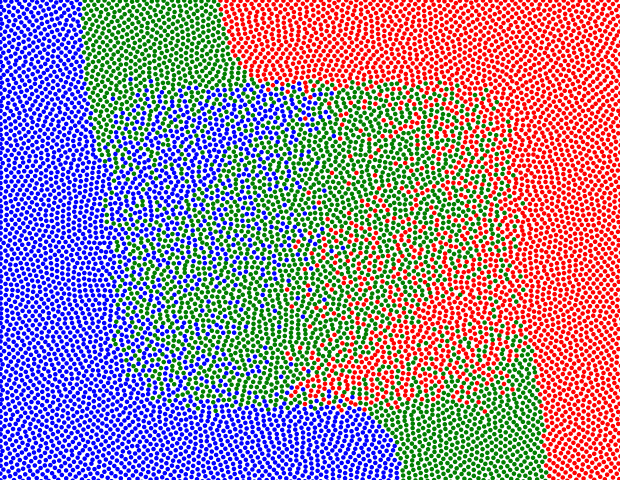}}&
  \subfloat[$t=0.5*t_{\mathrm{max}}$]
  {\includegraphics[width=.2\textwidth]{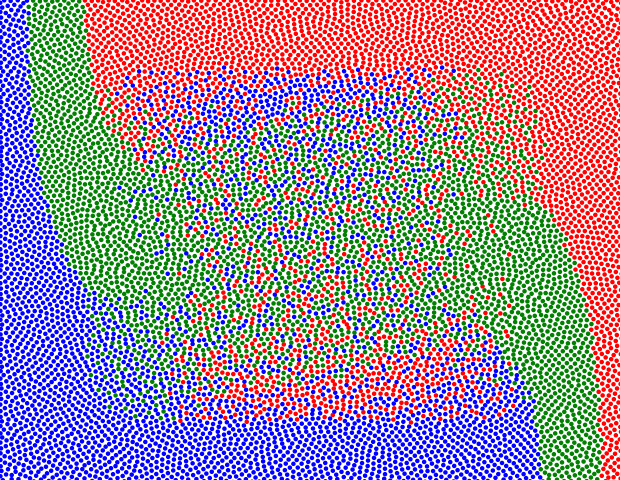}}&
  \subfloat[$t=0.75*t_{\mathrm{max}}$]
  {\includegraphics[width=.2\textwidth]{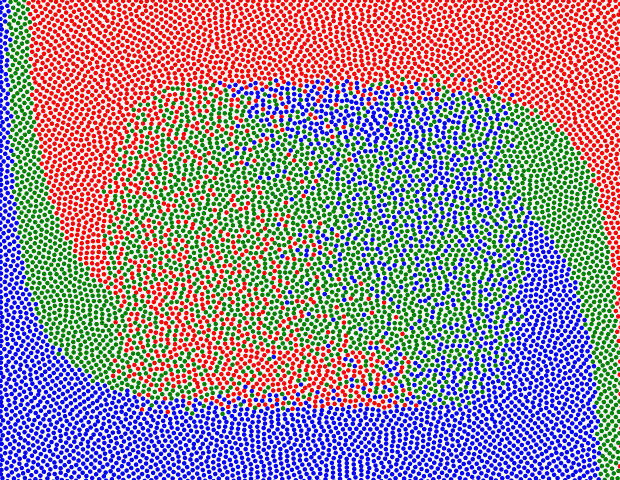}}&
  \subfloat[$t=t_{\mathrm{max}}=1.3$]
  {\includegraphics[width=.2\textwidth]{Results/Square-Reconstructed-Tmax=1.3/16.png}}\\
  \subfloat[$t=0.0$]
  {\includegraphics[width=.2\textwidth]{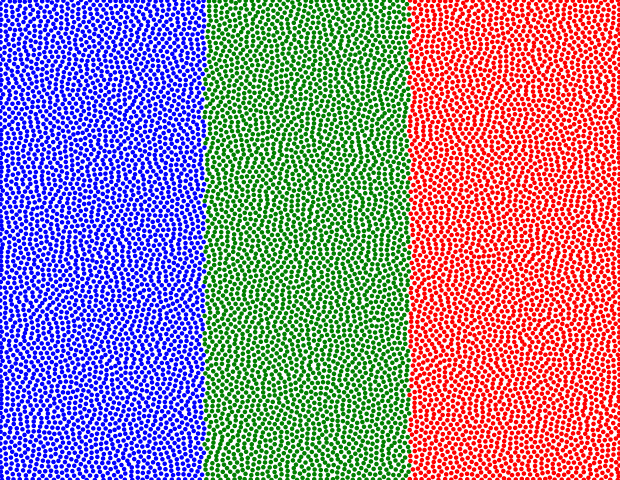}}&
  \subfloat[$t=0.25*t_{\mathrm{max}}$] 
  {\includegraphics[width=.2\textwidth]{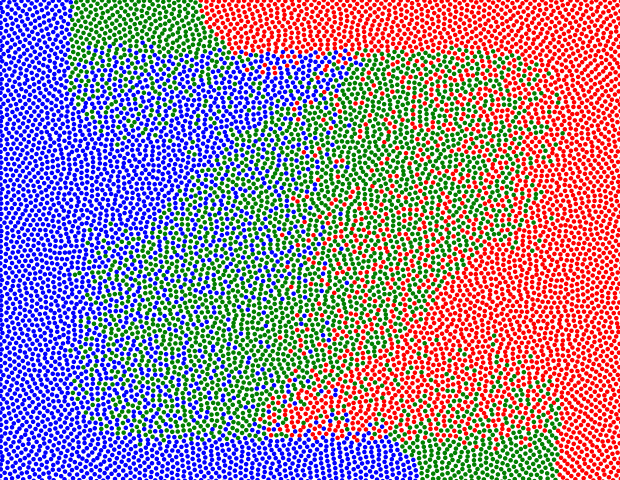}}&
  \subfloat[$t=0.5*t_{\mathrm{max}}$]
  {\includegraphics[width=.2\textwidth]{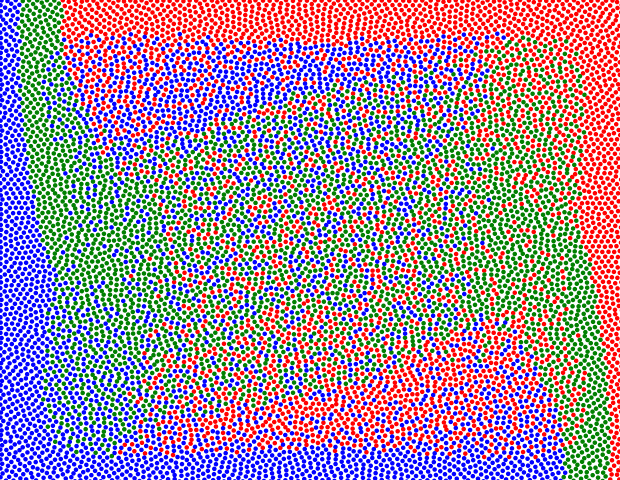}}&
  \subfloat[$t=0.75*t_{\mathrm{max}}$]
  {\includegraphics[width=.2\textwidth]{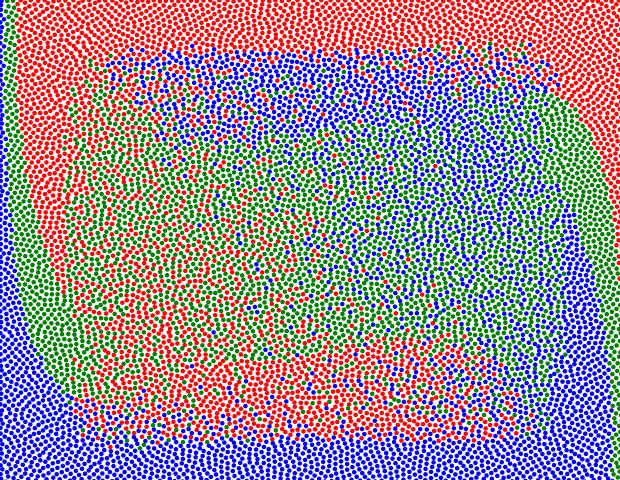}}&
  \subfloat[$t=t_{\mathrm{max}}=1.5$]
  {\includegraphics[width=.2\textwidth]{Results/Square-Reconstructed-Tmax=1.5/16.png}}
\end{tabular}}
\caption{\emph{(First row)} Beltrami flow in the unit square at
  various timesteps, a classical solution to Euler's equation. The
  color of the particles depend on their initial
  position. \emph{(Second to fifth row)} Generalized fluid flows that
  are reconstructed by our algorithm, using boundary conditions
  displayed in the first and last column. When $t_{\mathrm{max}} < 1$
  we recover the classical flow, while for $t_{\mathrm{max}} \geq 1$ the
  solution is not classical any more and includes some mixing. \label{fig:Square-Reconstructed}}
\end{figure}

\begin{figure}
\centering
\resizebox{\textwidth}{!} 
{
\setlength\tabcolsep{2 pt}
\begin{tabular}{@{}ccccc@{}}
  \subfloat[$t=0.0$]
  {\includegraphics[width=.2\textwidth]{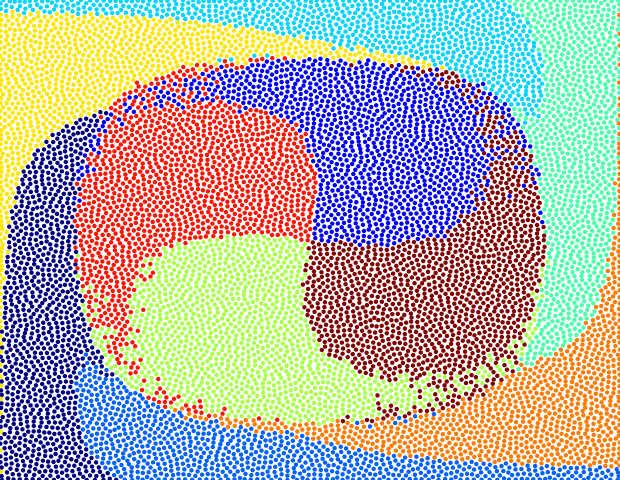}}&
  \subfloat[$t=0.125t_{\mathrm{max}}$]
  {\includegraphics[width=.2\textwidth]{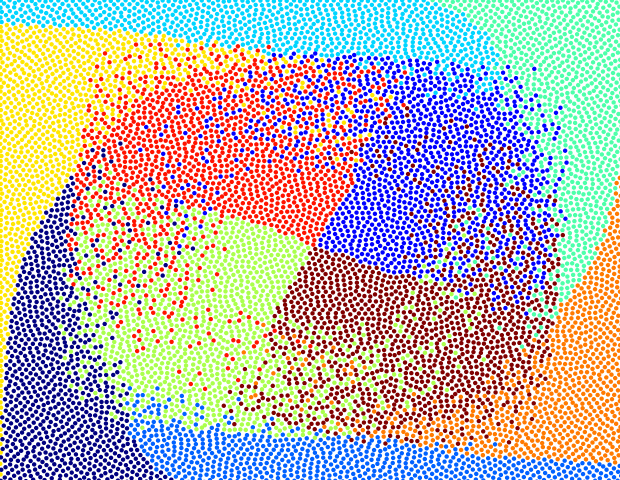}}&
  \subfloat[$t=0.25t_{\mathrm{max}}$]
  {\includegraphics[width=.2\textwidth]{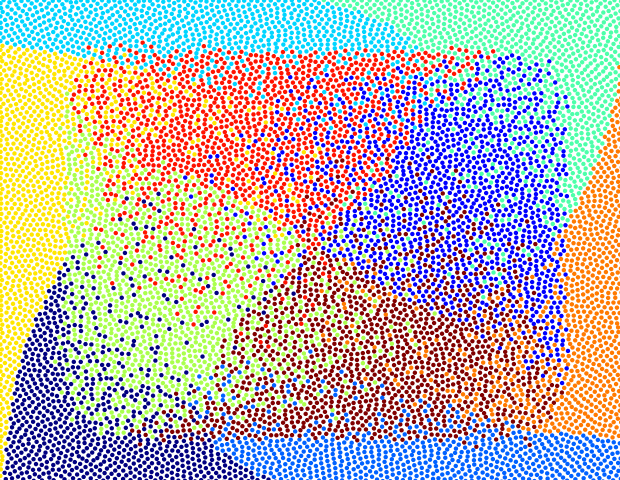}}&
  \subfloat[$t=0.375t_{\mathrm{max}}$]
  {\includegraphics[width=.2\textwidth]{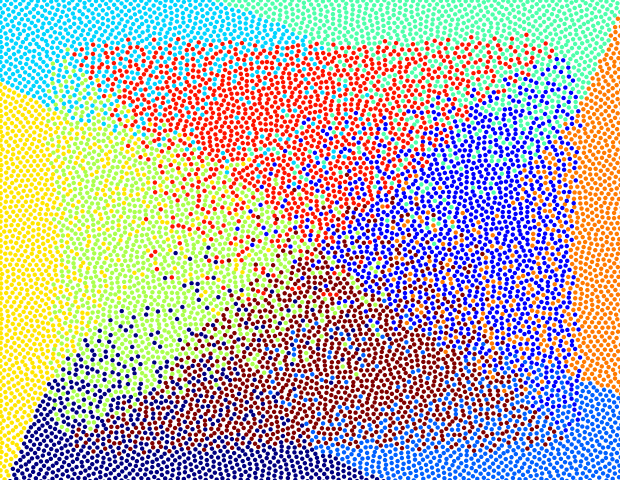}}&
  \subfloat[$t=0.5t_{\mathrm{max}}$]
  {\includegraphics[width=.2\textwidth]{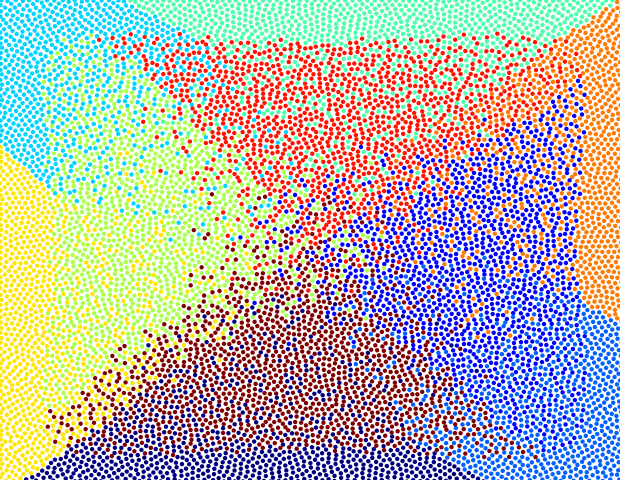}}\\
  \subfloat[$t=0.625t_{\mathrm{max}}$]
  {\includegraphics[width=.2\textwidth]{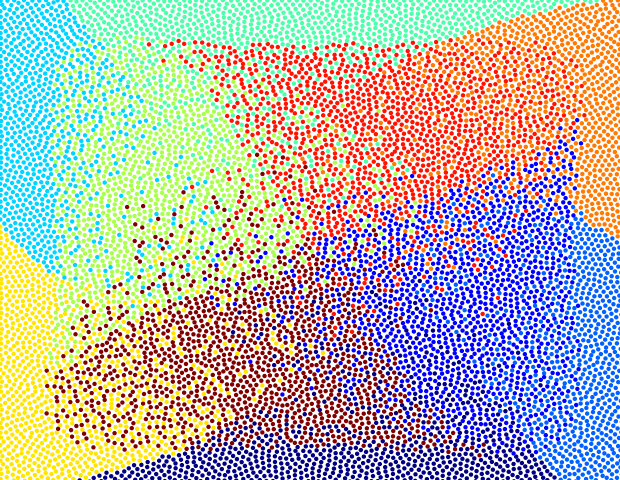}}&
  \subfloat[$t=0.75t_{\mathrm{max}}$]
  {\includegraphics[width=.2\textwidth]{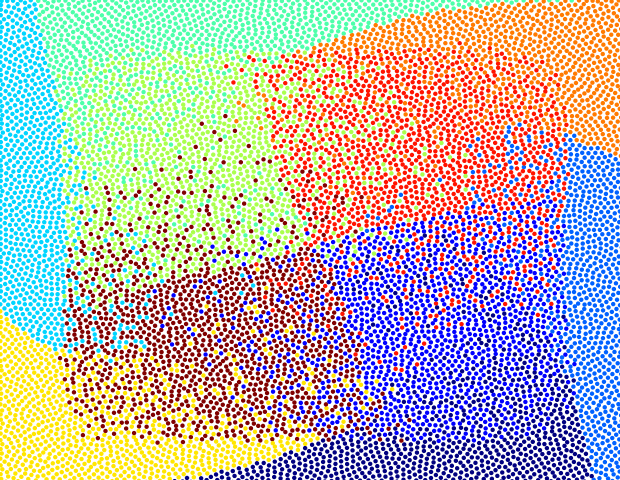}}&
  \subfloat[$t=0.875t_{\mathrm{max}}$]
  {\includegraphics[width=.2\textwidth]{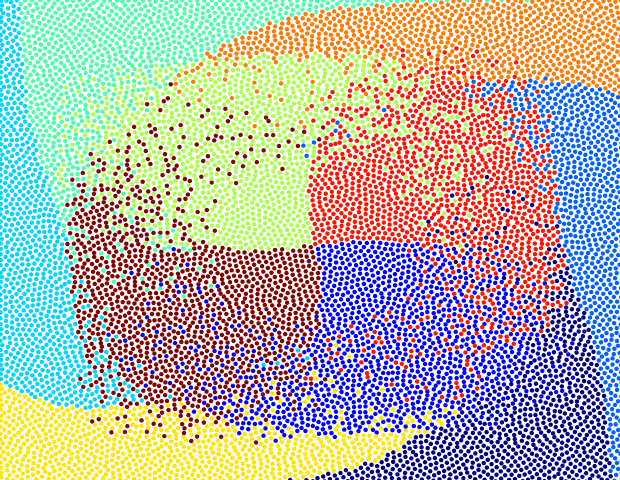}}&
  \subfloat[$t=t_{\mathrm{max}}=1.5$]
  {\includegraphics[width=.2\textwidth]{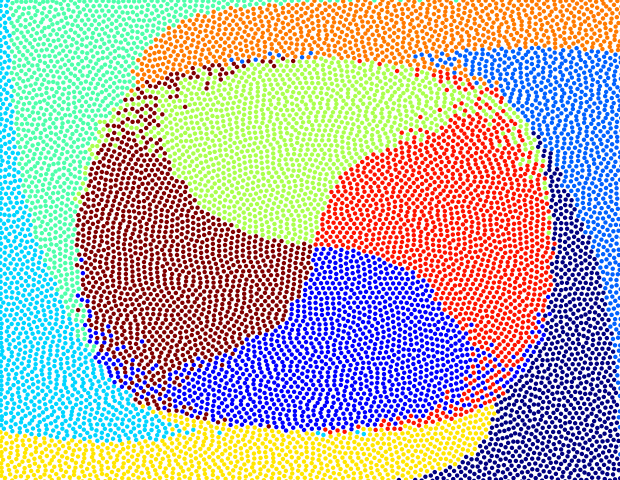}}&
\end{tabular}}
\caption{Using the $k$-means cluster algorithm, we cluster the
  reconstructed trajectories for the Beltrami flow in the square with
  $t_{\mathrm{max}} = 1.5$ into $10$ groups. This suggests that close
  to the boundary of the square the movement of particle is clockwise
  and deterministic while in the interior the movement is highly
  non-deterministic and counter-clockwise. \label{fig:Square-Reconstructed-Clusters}}
\end{figure}


\begin{figure}
\centering
\resizebox{.9\textwidth}{!} 
{
\setlength\tabcolsep{2 pt}
\begin{tabular}{@{}ccc@{}}
  \subfloat[$t=0.0$]
  {\includegraphics[width=.33\textwidth]{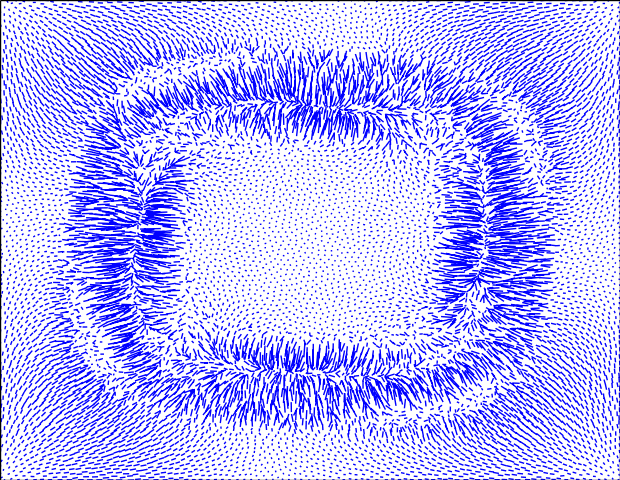}}&
  \subfloat[$t=0.125t_{\mathrm{max}}$]
  {\includegraphics[width=.33\textwidth]{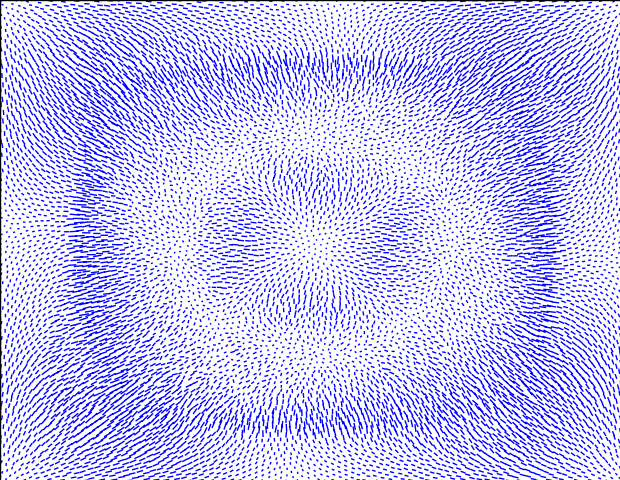}}&
  \subfloat[$t=0.25t_{\mathrm{max}}$]
  {\includegraphics[width=.33\textwidth]{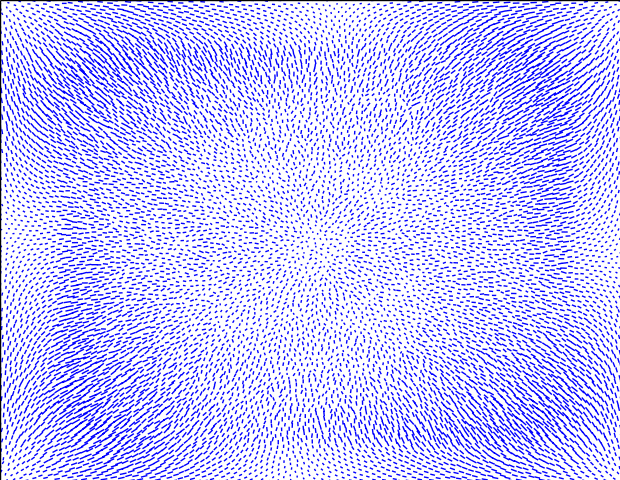}}\\
  \subfloat[$t=0.375t_{\mathrm{max}}$]
  {\includegraphics[width=.33\textwidth]{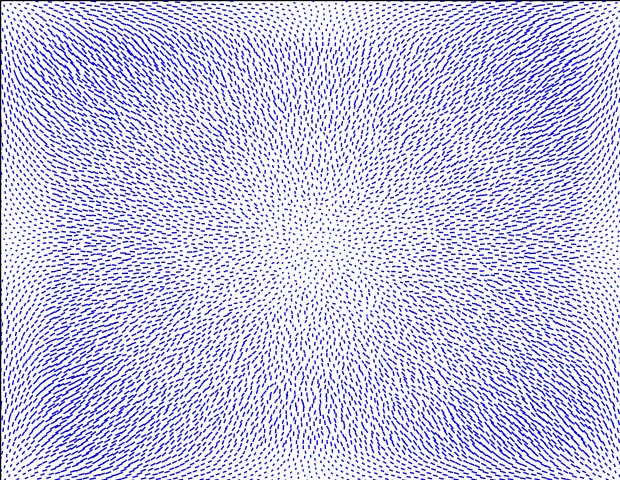}}&
  \subfloat[$t=0.5t_{\mathrm{max}}$]
  {\includegraphics[width=.33\textwidth]{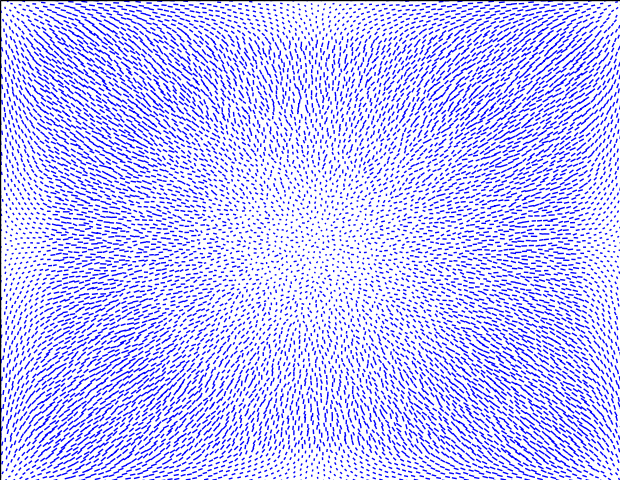}}&
  \subfloat[$t=0.625t_{\mathrm{max}}$]
  {\includegraphics[width=.33\textwidth]{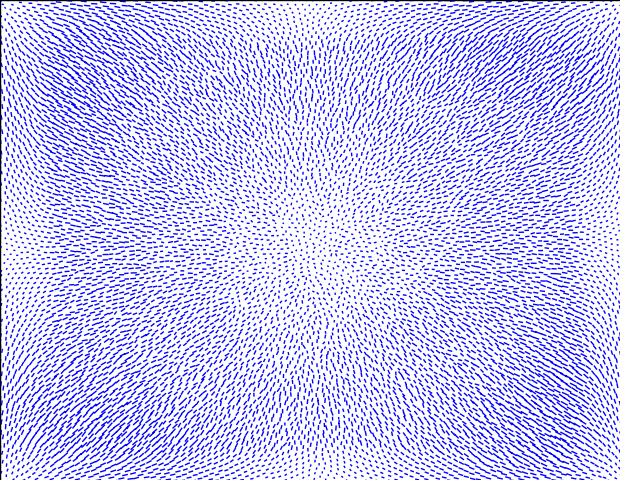}}\\
  \subfloat[$t=0.75t_{\mathrm{max}}$]
  {\includegraphics[width=.33\textwidth]{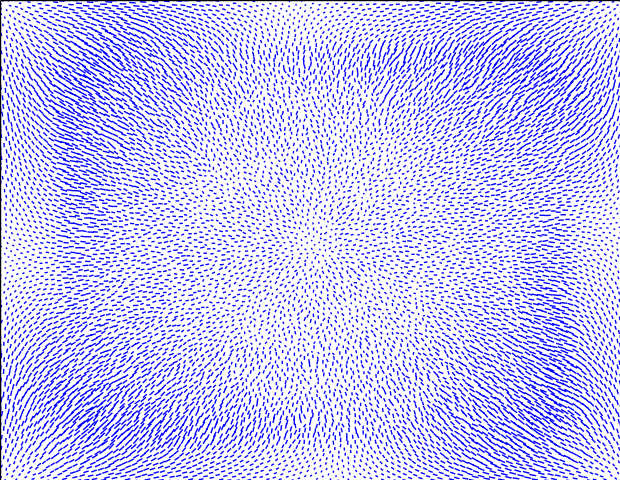}}&
  \subfloat[$t=0.875t_{\mathrm{max}}$]
  {\includegraphics[width=.33\textwidth]{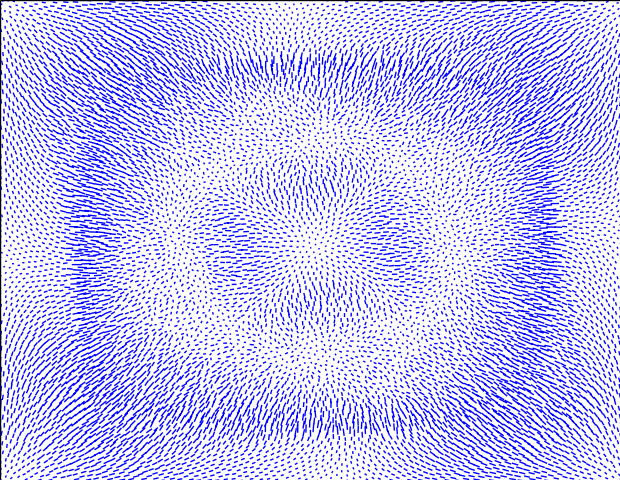}}&
  \subfloat[$t=t_{\mathrm{max}}=1.5$]
  {\includegraphics[width=.33\textwidth]{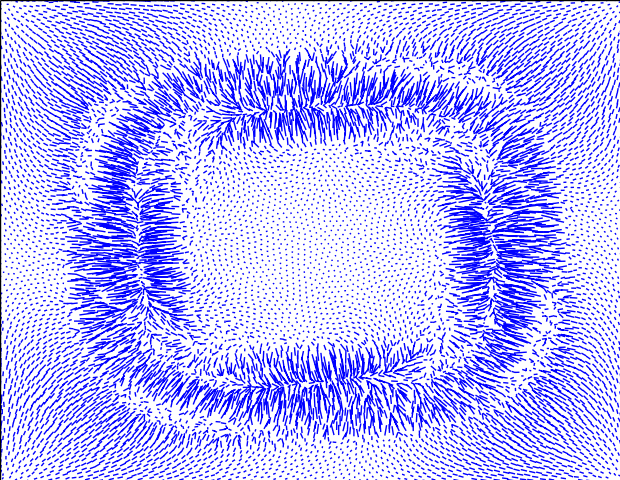}}
\end{tabular}}
\caption{Estimated pressure gradient for the Beltrami flow on the square with $t_{\mathrm{max}} = 1.5$.
\label{fig:Square-Reconstructed-Pressure}}
\end{figure}


\begin{figure}
\centering
\resizebox{\textwidth}{!} 
{
\setlength\tabcolsep{2 pt}
\begin{tabular}{@{}ccc@{}}
  \subfloat[$(x,y) = (-0.7,0)$]
  {\includegraphics[width=.33\textwidth]{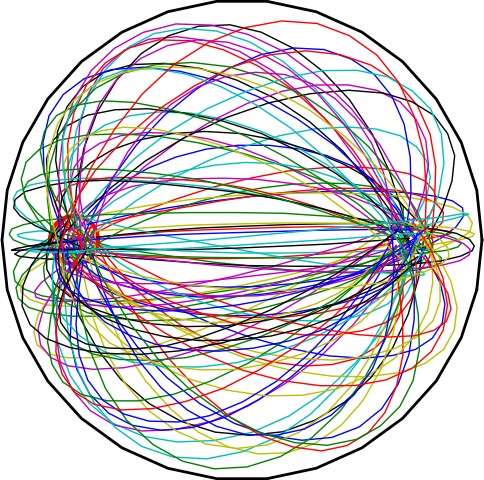}}&
  \subfloat[$(x,y) = (-0.35,0)$]
  {\includegraphics[width=.33\textwidth]{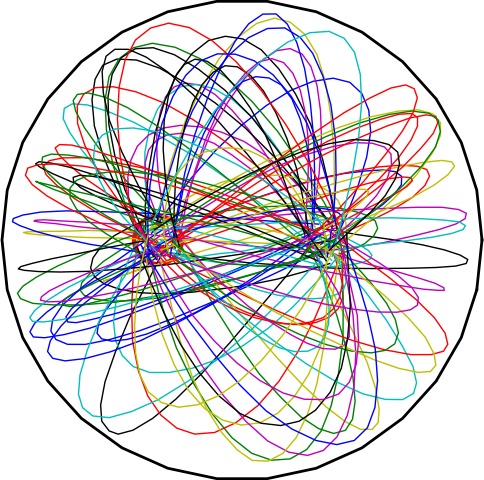}}&
  \subfloat[$(x,y) = (0,0)$]
  {\includegraphics[width=.33\textwidth]{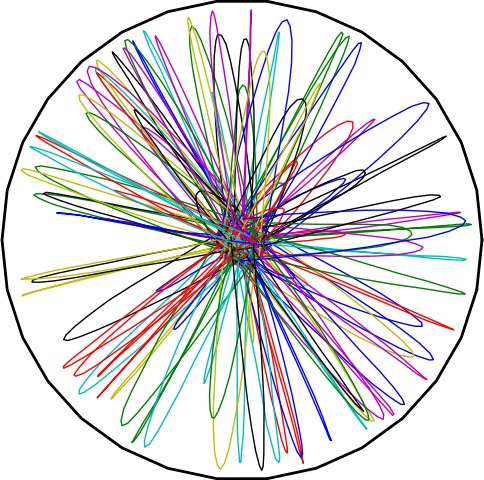}}\\
  \subfloat[$(x,y) = (0.2,0)$]
  {\includegraphics[width=.33\textwidth]{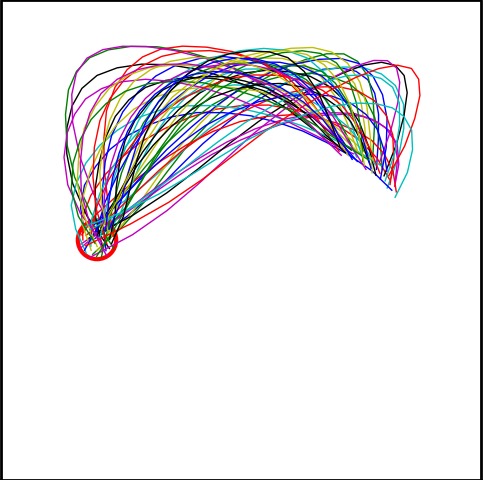}}&
  \subfloat[$(x,y) = (0.35,0)$]
  {\includegraphics[width=.33\textwidth]{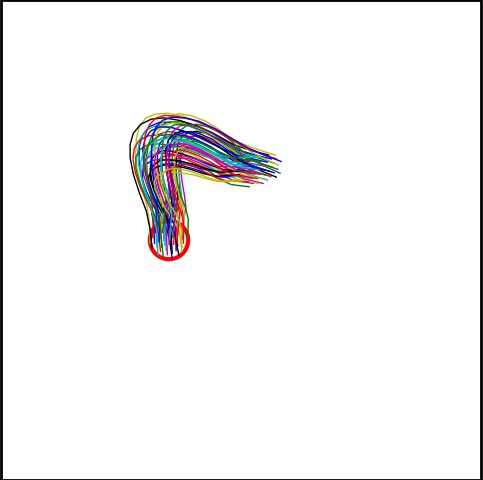}}&
  \subfloat[$(x,y) = (0.5,0)$]
  {\includegraphics[width=.33\textwidth]{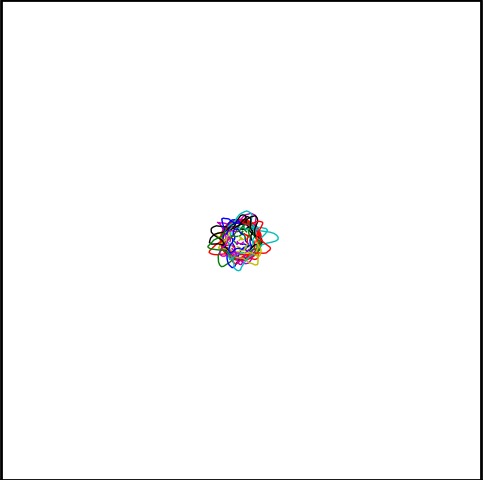}}
\end{tabular}}
\caption{We select particles whose initial position lie in a small
  disk, and display their trajectories according to the computed
  solution to \eqref{eq:DiscreteEnergy}. (Top) For the inversion of
  the unit disk (Bottom) For the Beltrami flow on the square, with
  $t_{\mathrm{max}} = 1.5$. \label{fig:Reconstructed-Trajectories}}
\end{figure}

\begin{figure}
\centering
\resizebox{\textwidth}{!} 
{
\includegraphics[width=.5\textwidth]{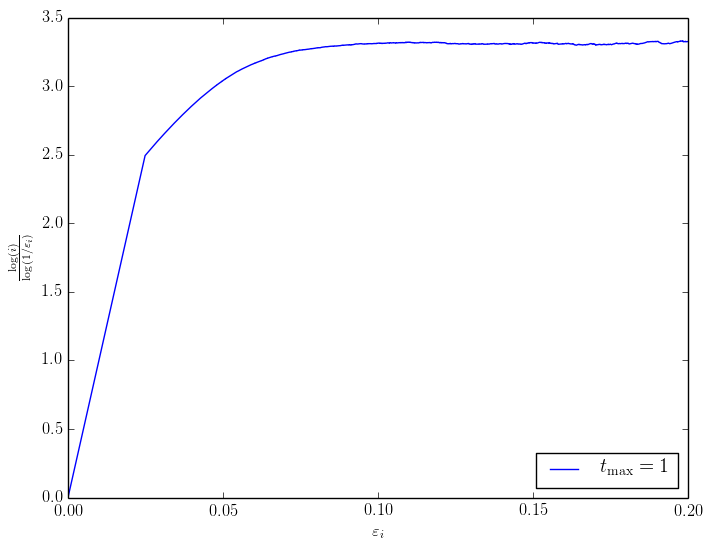}
\includegraphics[width=.5\textwidth]{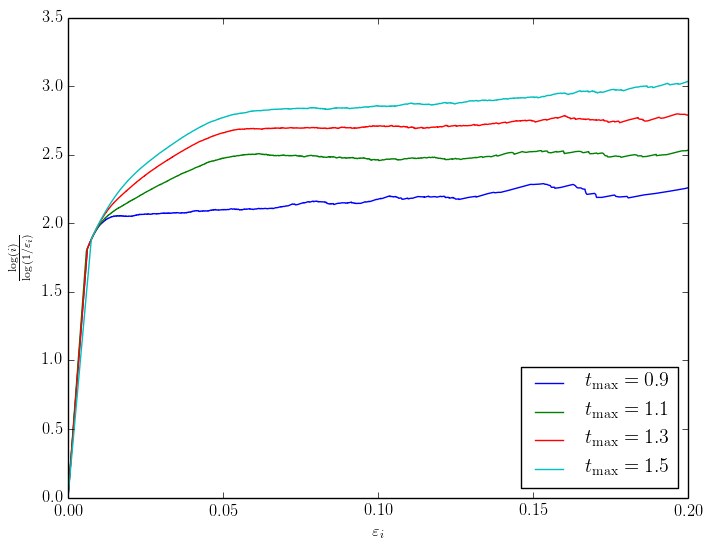}
}
\caption{Estimation of the box counting dimension of the support of
  the computed solution, see \S\ref{subsubsec:GDA}. \emph{(Left)} for
  the inversion of the unit disk \emph{(Right)} Comparison between the
  estimated box counting dimensions of the solutions to the Beltrami
  flow on the square, depending on the maximum time. \label{fig:Boxcount}}
\end{figure}


\bibliographystyle{alpha}
\bibliography{Include/AllPapers,Include/CGAL}


\end{document}